\documentclass[12pt,a4paper, final]{amsart}
\usepackage{amssymb}
\usepackage{hyperref}
\usepackage[notcite, notref]{showkeys}
\usepackage[utf8]{inputenc}

\usepackage{enumitem}
\setlist[enumerate]{font=\normalfont}

\usepackage{color}

\usepackage{listings} 
\lstdefinelanguage{Sage}[]{Python}
{morekeywords={False,sage,True},sensitive=true}
\definecolor{dblackcolor}{rgb}{0.0,0.0,0.0}
\definecolor{dbluecolor}{rgb}{0.01,0.02,0.7}
\definecolor{dgreencolor}{rgb}{0.2,0.4,0.0}
\definecolor{dgraycolor}{rgb}{0.30,0.3,0.30}

\newlist{enumA}{enumerate}{1}
\setlist[enumA]{label=(\alph*),ref=\alph*, font=\normalfont}

\newlist{enumN}{enumerate}{1}
\setlist[enumN]{label=(\arabic*),ref=\arabic*, font=\normalfont}

\newlist{enumR}{enumerate}{1}
\setlist[enumR]{label=(\roman*),ref=\roman*,font=\normalfont}

\theoremstyle{plain}
\newtheorem{theorem}{Theorem}[section]
\newtheorem{proposition}[theorem]{Proposition}

\newtheorem*{claim}{Claim}

\theoremstyle{definition}

\newtheorem{example}[theorem]{Example}
\newtheorem{remark}[theorem]{Remark}

\newtheorem{notation}[theorem]{Notation}

\theoremstyle{remark}

{}

\numberwithin{equation}{section}
\newcommand{\N}{\mathbb N}
\newcommand{\Z}{\mathbb Z}

\newcommand{\R}{\mathbb R}
\newcommand{\C}{\mathbb C}

\newcommand{\fe}{\mathfrak e}
\newcommand{\ff}{\mathfrak f}
\newcommand{\fg}{\mathfrak g}
\newcommand{\fh}{\mathfrak h}

\newcommand{\fl}{\mathfrak l}

\newcommand{\fp}{\mathfrak p}

\newcommand{\ft}{\mathfrak t}

\DeclareMathOperator{\GL}{GL}

\DeclareMathOperator{\SU}{SU}
\DeclareMathOperator{\Sp}{Sp}

\DeclareMathOperator{\Spin}{Spin}

\newcommand{\gl}{\mathfrak{gl}}

\newcommand{\so}{\mathfrak{so}}
\newcommand{\spp}{\mathfrak{sp}}
\newcommand{\su}{\mathfrak{su}}

\newcommand{\op}{\operatorname}
\newcommand{\Id}{\textup{Id}}

\newcommand{\mi}{\mathrm{i}}

\newcommand{\inner}[2]{\langle {#1},{#2}\rangle }
\newcommand{\innerdots}{\langle {\cdot},{\cdot}\rangle }
\newcommand{\ee}{\varepsilon}

\DeclareMathOperator{\Sym}{Sym}

\DeclareMathOperator{\Ad}{Ad}
\DeclareMathOperator{\ad}{ad}
\DeclareMathOperator{\Cas}{Cas}

\DeclareMathOperator{\st}{st}

\DeclareMathOperator{\Spec}{Spec}

\DeclareMathOperator{\Iso}{Iso}

\newcommand{\kil}{\operatorname{B}}

\newcommand{\PP}{\mathcal{P}}

\title[Strongly isotropy irreducible spaces]{Strongly isotropy irreducible spaces whose Laplacians have all their eigenspaces irreducible}

\author{Emilio~A.~Lauret}
\address{Instituto de Matemática (INMABB), Departamento de Matemática, Universidad Nacional del Sur (UNS)-CONICET, Bahía Blanca, Argentina.}
\email{emilio.lauret@uns.edu.ar}
\author{Fiorela Rossi Bertone}
\address{Instituto de Matemática (INMABB), Departamento de Matemática, Universidad Nacional del Sur (UNS)-CONICET, Bahía Blanca, Argentina.}
\email{fiorela.rossi@uns.edu.ar}

\thanks{This research was supported by a grant (code PGI 24/ZL25) from SGCYT--UNS}
\date{\today}

\begin{document}

\begin{abstract}
Let $H$ be a closed subgroup of a compact connected Lie group $G$ such that $G/H$ is simply connected, the action of $G$ on $G/H$ is almost effective, and $G/H$ is a strongly isotropy irreducible space. 
The smallest positive eigenvalue of the Laplace-Beltrami operator $\Delta$ associated to the standard metric on $G/H$ was recently computed.  
In this article we obtain a second non-zero eigenvalue of $\Delta$. 

As a consequence, we prove that all eigenspaces of the Laplacian are irreducible as $G$-modules if and only if $(G,H)$ is a compact symmetric pair of real rank one, or $(G,H)$ is equal to $(\operatorname{G}_2,\operatorname{SU}(3))$ or $(\operatorname{Spin}(7),\operatorname{G}_2)$. 
\end{abstract} 

\maketitle

%\tableofcontents
	
\section{Introduction}\label{sec:intro}

A homogeneous space $M:=G/H$ ($H$ a closed subgroup of a Lie group $G$) is called \emph{isotropy irreducible} if the isotropy representation (i.e.\ $\chi:H\to \GL(T_{eH}M)\simeq\GL(\fg/\fh)$ given by $\chi(h)\cdot (X+\fh)=\Ad(h)\cdot X+\fh$) is irreducible. 
Similarly, $M$ is said to be \emph{strongly isotropy irreducible} if $\chi$ restricted to the identity connected component $H^0$ of $H$ is irreducible. 
Every strongly isotropy irreducible space is in particular isotropy irreducible, though the converse is not true. 
Of course, both notions coincide when $H$ is connected.

Strongly isotropy irreducible spaces were classified independently by 
Manturov~\cite{Manturov61a,Manturov61b,Manturov66},
Wolf~\cite{Wolf68-isot-irred, Wolf84-isot-irred-errata}
and Krämer \cite{Kramer-stronglyisotropyirredspaces}. 
They consist in 10 families and 33 isolated cases (see Tables~\ref{table:isotirred-families}--\ref{table:isotirred-isolated}). 
The remaining isotropy irreducible spaces were classified by Wang and Ziller~\cite{WangZiller91}.

These spaces have remarkable geometric properties. 
For instance, every $G$-invariant (Riemannian) metric on an isotropy irreducible space $M=G/H$ is a positive multiple of the standard one $g_{\st}$. 
Furthermore, $(G/H,g_{\st})$ is Einstein, that is, the Ricci tensor of $g_{\st}$ is proportional to $g_{\st}$. 
Every irreducible symmetric space is strongly isotropy irreducible. 
A non-compact isotropy irreducible space is symmetric (\cite[Thm.~7.46]{Besse}). 

Recent attention has been given to spectral properties of isotropy irreducible spaces:
\cite{CaoHe15}
\cite{Schwahn-Lichnerowicz},
\cite{ProvenzanoStubbe25},
\cite{LauretTolcachier-nu-stab},
\cite{LauretRossiTolcachier-Saloff-Coste}.

For an arbitrary compact connected homogeneous space $G/H$ with $G$ semisimple, the spectrum of the Laplace-Beltrami operator $\Delta$ associated to the standard metric $g_{\st}$ (i.e.\ the corresponding inner product on the orthogonal complement of $\fh$ in $\fg$ with respect to the Killing form $\kil_\fg$  is equal to $-\kil_\fg$) is given by (see e.g.\ \cite[Thm.~3.1]{LauretTolcachier-nu-stab}) the multiset 
\begin{equation}\label{eq:Spec(G/H,g_st)}
\Spec(G/H,g_{\st})
:=\Big\{\!\!\Big\{ 
	%\underbrace{
\lambda^\pi,\dots,\lambda^\pi
	%}_{\dim V_\pi\times \dim V_\pi^H\text{-times} } 
	: \pi\in\widehat G
\Big\}\!\!\Big\}
\qquad \big(\lambda^\pi\text{ repeated $(\dim V_\pi\times \dim V_\pi^H)\text{-times}$}\big)
.
\end{equation}
%with $\lambda^\pi$ reapeted $\dim V_\pi\times \dim V_\pi^H$-times. 
Here, $\widehat G$ denotes a complete representative set of isomorphism classes of irreducible complex representations $\pi:G\to\GL(V_\pi)$ of $G$ ($V_\pi$ is a complex vector space), 
$\lambda^\pi$ is determined by $d\pi(-\Cas_\fg)=\lambda^\pi \,\Id_{V_\pi}$, and $\Cas_\fg$ is the Casimir element of $\fg$ (i.e.\ $\Cas_\fg=\sum_j X_j^2\in\mathcal U(\fg_\C)$, where $\{X\}_j$ is any orthonormal basis of $\fg$ with respect to $-\kil_\fg$). 
The scalar $\lambda^\pi$ can be easily computed in terms of the highest weight $\Lambda_\pi$ of $\pi$, once a maximal torus $T$ of $G$ and a positive root system of $(\fg_\C,\ft_\C)$ have been fixed, which is assumed from now on. 

Note that $\pi\in\widehat G$ contributes with an eigenvalue in $\Spec(G/H,g_{\st})$ only if $\pi$ lies in 
\begin{equation} \label{eq:hatG_H}
\widehat G_H:=\{\pi\in\widehat G: V_\pi^H\neq0\},
\end{equation}
the so called \emph{spherical representations} of $(G,H)$. 
Actually, the main difficulty in describing  $\Spec(G/H,g_{\st})$ is to characterize $\widehat G_H$. 
It is well known that 
\begin{equation}\label{eq:P^+(G,H)}
\PP^+(G,H):=\{\Lambda\in \PP^+(G): \pi_{\Lambda}\in \widehat G_H\}
\end{equation}is a monoid. 
More precisely, if $\Lambda$ and $\Lambda'$ are two dominant $G$-integral weights of $G$ satisfying $\pi_{\Lambda},\pi_{\Lambda'}\in\widehat G_H$, then $\pi_{\Lambda+\Lambda'}\in\widehat G_H$.

Although $\widehat G_H$, or equivalently $\PP^+(G,H)$, is explicitely described for compact symmetric pairs and many other cases (e.g.\ for Gelfand pairs by Krämer~\cite{Kramer-sphericalpairs}), it is far from being known in general. 
In particular, to the best of the authors' knowledge, no explicit description of $\widehat G_H$ is known for any non-symmetric strongly isotropy irreducible space $G/H$, except for two very special cases, namely, when $G/H$ is covered by $\op{G}_2/\SU(3)$ or by $\Spin(7)/\op{G}_2$.
These two spaces endowed with the standard metrics are isometric to round spheres of dimension six and seven respectively. 
Furthermore, they are the only Gelfand pairs (or spherical pairs) among non-symmetric strongly isotropy irreducible spaces. 

For $G/H$ a simply connected strongly isotropy irreducible space, the authors in collaboration with Alejandro Tolcachier computed explicitely the smallest positive eigenvalue of $\Spec(G/H,g_{\st})$ (see \cite[\S3]{LauretRossiTolcachier-Saloff-Coste}), namely,
\begin{equation}
\lambda_1(G/H,g_{\st})
:= \min \left\{\lambda^\pi: \pi\in\widehat G_H\smallsetminus\{1_G\}\right\}
. 
\end{equation}
Tables~\ref{table:isotirred-families}--\ref{table:isotirred-isolated} show the highest weight $\Lambda_1$ satisfying $\lambda_1(G/H,g_{\st})=\lambda^{\pi_{\Lambda_1}}$. 
Moreover, the main result of \cite{LauretRossiTolcachier-Saloff-Coste} states that 
\begin{equation}
1<\frac{\lambda_1(G/H,g_{\st})}{E(G/H,g_{\st})}\leq 16 
\end{equation}
for an arbitrary strongly isotropy irreducible space $G/H$, where $E(G/H,g_{\st})$ stands for the Einstein constant of $(G/H,g_{\st})$.

The first contribution of this paper is a second explicit eigenvalue of $\Spec(G/H,g_{\st})$ for each simply connected strongly isotropy irreducible space, namely,  $\lambda^{\pi_{\Lambda_2}}$ as shown in Tables~\ref{table:isotirred-families}--\ref{table:isotirred-isolated}. 
This new eigenvalue is essentially new in the sense that $\{\Lambda_1,\Lambda_2\}$ are linearly independent, so $a_1\Lambda_1+a_2\Lambda_2\in\PP^+(G,H)$ for $a_1,a_2\in\Z_{\geq0}$ are all distinct. 
This is done in \S\ref{sec:Lambda_2}.

We next explain a consequence of the above fact, which motivated this paper. 
Let $(G/H,g)$ be an arbitrary compact homogeneous Riemannian manifold. 
The $G$-invariant metric $g$ on $M:=G/H$ corresponds to an $\Ad(H)$-invariant inner product $\innerdots_g$ on a fixed reductive complement $\fp$ of $\fh$ in $\fg$ (i.e.\ $\fg=\fh\oplus\fp$ and $\Ad(h)\cdot\fp\subset\fp$ for all $h\in H$). 
For each $a\in G$, let $\varphi_a:M\to M$ denote the map given by $\varphi_a(xH)=axH$. 

Since the Laplace-Beltrami operator $\Delta$ associated to $(M,g)$ commutes with isometries and $\varphi_a\in\Iso(M,g)$ for every $a\in G$, 
we have that 
$\Delta_g  (f\circ \varphi_a) = \Delta_g(f)\circ \varphi_a$. 
The left regular representation of $G$ on $C^\infty(M,\R)\subset L^2(M,\R)$ given by $a\cdot f=f\circ \varphi_{a^{-1}}$ preserves the eigenspaces of $\Delta$. 
Indeed, if $\Delta(f)=\lambda\, f$ for $f\in C^\infty(M,\R)$ and $\lambda\in \Spec(G/H,g)$, then 
\begin{align*}
\Delta (a\cdot f)&
= \Delta (f\circ \varphi_{a^{-1}}) 
= \Delta(f)\circ \varphi_{a^{-1}}
= (\lambda \, f)\circ \varphi_{a^{-1}}
= \lambda \, (f\circ \varphi_{a^{-1}})
= \lambda \, (a\cdot f)
\quad\forall\, a\in G.
\end{align*}
Consequently, every eigenspace 
\begin{equation}
\op{Eig}(\lambda):=\{f\in C^\infty(M,\R): \Delta(f)=\lambda\, f\}
\qquad (\lambda\in\Spec(G/H,g))
\end{equation}
of $\Delta$, which is always of finite dimension, is a real $G$-module, or equivalently, a real representation of $G$.

Berger, Gauduchon and Mazet~\cite[Prop.~C.I.8]{BergerGauduchonMazet-LNM} proved that if $H^0$ acts transitively on the unit sphere $\{X\in\fp: \inner{X}{X}_g=1\}$ of $\fp$, then all eigenspaces of the Laplace-Beltrami operator of $(G/H,g)$ are $G$-irreducible. 
The above condition holds for each compact rank one symmetric space (CROSS), and additionally for the  (non-symmetric) strongly isotropy irreducible spaces covered by $\op{G}_2/\SU(3)$ and $\Spin(7)/\op{G}_2$. 

More recently, for any irreducible compact symmetric space $G/K$ of real rank at least two, 
Petrecca and Röser~\cite[Thm.~4]{PetreccaRoeser18} found one eigenspace of the Laplacian that is not $G$-irreducible. 

Combining the last two mentioned results it follows (see \cite[Thm.~4]{PetreccaRoeser18}), for a given irreducible compact symmetric space $G/K$, that 
\begin{equation}
\text{$\op{Eig}(\lambda)$ is $G$-irreducible for all $\lambda\in\Spec(G/K,g_{\st})$ if and only if $G/K$ is a CROSS.}
\end{equation}
The main result of this article extends this characterization to simply connected strongly isotropy irreducible spaces.

\begin{theorem}\label{thm1:main}
Let $G/H$ be a compact simply connected strongly isotropy irreducible space, where $G$ is a compact connected Lie group whose action on $G/H$ is almost effective. 
Then, all eigenspaces of the Laplace-Beltrami operator associated to $(G/H,g_{\st})$ are $G$-irreducible 
if and only if 
$G/H$ is a CROSS, or $(G,H)$ equals $(\op{G}_2,\SU(3))$ or $(\Spin(7),\op{G}_2)$. 
\end{theorem}

In words, Theorem~\ref{thm1:main} ensures that the only simply connected strongly isotropy irreducible spaces $G/H$ satisfying that $\op{Eig}(\lambda)$ is $G$-irreducible for all $\lambda\in\Spec(G/H,g_{\st})$, are those found by Berger, Gauduchon and Mazet \cite[Prop.~C.I.8]{BergerGauduchonMazet-LNM}.

The proof of Theorem~\ref{thm1:main} follows from Theorem~\ref{thm:standard-fails}, which provides a sufficient condition for an arbitrary standard homogeneous space $(G/H,g_{\st})$ to have a Laplacian with at least one $G$-reducible eigenspace.
This sufficient condition requires the existence of two linearly independent elements $\Lambda_1,\Lambda_2$ in $\PP^+(G,H)$, as the one given in Tables~\ref{table:isotirred-families}--\ref{table:isotirred-isolated}. 
In particular, the proof of Theorem~\ref{thm1:main} gives an alternative and shorter proof for \cite[Thm.~4]{PetreccaRoeser18}, proving that any simply connected irreducible symmetric space of real rank $\geq2$ has at least one $G$-reducible eigenspace.

\subsection*{Disclosure}

The proofs in this article were discovered, as in `the good old days', without the help of AI. 
Claude was used only to improve the wording in a few places. 
This will surely be the authors' last paper made in this way.

%\subsection*{Acknowledgments}

\section{Preliminaries}\label{sec:preliminaries}

In this subsection we set some Lie theoretical conventions and notations, as well as we recall some results.
For further details, see \cite[\S2]{LauretTolcachier-nu-stab} and \cite[\S2]{LauretRossiTolcachier-Saloff-Coste}, from which the material presented here was extracted.
The reader is invited to skip this subsection initially and return to it as needed.

Let $G$ be a compact semisimple Lie group with Lie algebra $\fg$. 
Let $T$ be a maximal torus of $G$, so $\ft_\C$ is a Cartan subalgebra of $\fg_\C$. 
We pick a positive root systems $\Phi^+(\fg_\C,\ft_\C)$, and we denote by $\omega_1,\omega_2,\dots$ the corresponding fundamental weights.
The lattice $\PP(\fg_\C)$ of algebraically integral weights can be written as $\PP(\fg_\C):=\bigoplus_{j} \Z\omega_j$, and its subset of dominant elements is $\PP^+(\fg_\C)=\bigoplus_{j} \Z_{\geq0}\omega_j$. 

All representations of $G$ will be assumed complex and finite dimensional, unless otherwise explicitly stated.
Given $\pi$ a representation of $G$, we denote by $V_\pi$ its underlying $\C$-vector space, i.e.\  $\pi:G\to\GL(V_\pi)$ is a continuous morphism. 

The highest weight theorem (see e.g.\ \cite[Thm.~5.5]{Knapp-book-beyond}) tells us that $\widehat G$ is in correspondence with $\PP^+(G):=\PP(G)\cap \PP^+(\fg_\C)$, where $\PP(G)$ denotes the sublattice of $\PP(\fg_\C)$ given by $G$-integral weights. 
For $\Lambda\in\PP^+(G)$, we denote by $\pi_{\Lambda}$ the (unique up to equivalence) irreducible representation of $G$ with highest weight $\Lambda$. 
Sometimes, for $\Lambda\in\PP^+(\fg_\C)$, we will also denote by $\pi_{\Lambda}$ the irreducible representation of $\fg_\C$ with highest weight $\Lambda$, even if $\Lambda$ is not $G$-integral. 

Let $\kil_\fg$ denote the Killing form of $\fg$, and let $\kil_\fg^*$ be its extension to $\ft_\C^*$. 
For $\Lambda\in\PP^+(\fg_\C)$, the scalar $\lambda^{\pi_{\lambda}}$ determined by $\pi_{\Lambda}(\Cas_{\fg}) = -\lambda^{\pi_{\Lambda}}\, \Id_{V_\pi}$ is given by (see \cite[Prop.~5.28]{Knapp-book-beyond})
\begin{equation}\label{eq:Casimir} 
\lambda^{\pi_{\lambda}} = \kil_\fg^* (\Lambda+2\rho_\fg,\Lambda),
\end{equation}
where $\rho_{\fg}$ is half the sum of all positive roots. 

Let $H$ be a closed subgroup of $G$.
We denote by $\fh$ the corresponding Lie subalgebra of $\fg$. 
We will always assume that the maximal torus $T$ of $G$ satisfies that $T\cap H^0$ is a maximal torus of $H^0$ (the trivial connected component of $H$).  
Note that $\fh$ is not necessarily semisimple. 

We pick a compatible positive root system associated to the maximal abelian subalgebra $(\fh\cap\ft)_\C$ of $\fh_\C$. 
Let $\eta_1,\eta_2,\dots,$ denote the corresponding fundamental weights. 
The irreducible representations of $H$ will be denoted by $\sigma$.

\begin{notation}\label{not:irrep-G-H}
For complex simple Lie algebras, we use Bourbaki's choice (see \cite[\S{}VI.4]{Bourbaki-Lie4-6}) for the ordering of simple roots, fundamental weights (see also \cite[Table~1]{LauretTolcachier-nu-stab}), and the basis of the space containing the dual of the Cartan subalgebra. 
\end{notation}

Sometimes we will have $\fh=\fh_1\oplus\fh_2$, with $\fh_1,\fh_2$ ideals of $\fh$. 
Any irreducible representation of $\fh$ is of the form $\sigma\widehat\otimes\sigma'$, with $\sigma$ (resp.\ $\sigma'$) an irreducible representation of $\fh_1$ (resp.\ $\fh_2$). 
Here $\widehat\otimes$ stands for the outer tensor product. 
Similarly, the highest weight of $\sigma\,\widehat\otimes\, \sigma'$ is $\zeta+\zeta'$, with $\zeta$ (resp.\ $\zeta'$) the highest weight of $\sigma$ (resp.\ $\sigma'$). 
In this situation we will denote by $\eta_1,\eta_2,\dots$ (resp.\ $\eta_1',\eta_2',\dots$) the fundamental weights of $\fh_1$ (resp.\ $\fh_2$).

We conclude the section with some formulas that will be used for $\fg$ and $\fh$. 
For this reason, we denote by $\fl$ an arbitrary compact Lie algebra, and $\fl_\C$ stands for its complexification.

\begin{remark}[Exterior and symmetric products of the standard representation] \label{rem:ExtSym-standard}
Let $\theta$ be a finite dimensional representation of $\fl_\C$. 
We denote by $\bigwedge^p(\theta)$  (resp.\ $\Sym^k(\theta)$) the $p$-exterior (resp.\ the $k$-symmetric) power of $\theta$, with underlying vector space $\bigwedge^p (V_\theta)$ (resp.\ $\Sym^k(V_\theta)$), for any integer $0\leq p\leq \dim V_\theta$ (resp.\ $k\geq0$). 
Notice that $\Lambda^0(\theta) =\Sym^0(\theta) =1_{\fl_\C}$ and $\Lambda^1(\theta) =\Sym^1(\theta) =\theta$.

If $\fl_\C$ is simple of type $A_n$, we have that $\pi_{\omega_p}\simeq \textstyle{\bigwedge^p}(\pi_{\omega_1})$,  $\pi_{k\omega_1}\simeq \Sym^k(\pi_{\omega_1})$, and $\pi_{\omega_p}^*\simeq\pi_{\omega_{n+1-p}}$ for all integers $k\geq1$, $1\leq p\leq n$. 

If $\fl_\C$ is simple of type $B_n$ or $D_n$, $\omega_p=\ee_1+\dots+\ee_p$ if $1\leq p\leq n-1$ for type $B_n$ and if $1\leq p\leq n-2$ for type $D_n$, 
$\Sym^k(\pi_{\omega_1})\simeq \pi_{k\omega_1}\oplus \Sym^{k-2}(\pi_{\omega_1})$ 
for all integers $k\geq2$,
and 
\begin{equation*}
\textstyle{\bigwedge^p}(\pi_{\omega_1})
\simeq 
\begin{cases}
\pi_{\ee_1+\dots+\ee_p} &\quad\text{if }1\leq p\leq n-1, \\
\pi_{\ee_1+\dots+\ee_n} &\quad\text{if }p=n\text{ and $\fl_\C$ is of type $B_n$}, \\
\pi_{\ee_1+\dots+\ee_n}\oplus \pi_{\ee_1+\dots+\ee_{n-1}-\ee_n} &\quad\text{if }p=n\text{ and $\fl_\C$ is of type $D_n$}.
\end{cases}
\end{equation*}

If $\fl_\C$ is simple of type $C_n$,  $\textstyle{\bigwedge^p}(\pi_{\omega_1}) \simeq \pi_{\omega_p}\oplus \textstyle{\bigwedge^{p-2}}(\pi_{\omega_1})$, and
$\Sym^k(\pi_{\omega_1})\simeq \pi_{k\omega_1}$ 
for all integers $k\geq0$, $1\leq p\leq n$. 
\end{remark}

\begin{remark}\label{rem:rhoxrho^*}
Assume that $\fl_\C$ is simple. 
For irreducible representations $\theta_1,\theta_2$ of $\fl_\C$, it is well known that $\theta_1\otimes\theta_2$ contains the trivial representation $1_{\fl_\C}$ of $\fl_\C$ if and only if $\theta_1^*\simeq\theta_2$.
In this case, $1_{\fl_\C}$ occurs exactly once, which is equivalent to $\dim (V_{\theta_1}\otimes V_{\theta_1^*})^{\fl_\C}=1$. 
Moreover, when $\theta:=\theta_1=\theta_2$, 
$\theta\otimes\theta \simeq \Sym^2(\theta) \oplus {\textstyle\bigwedge^2}(\theta)$ and, 
	$1_{\fh_\C}$ occurs in $\Sym^2(\theta)$ if and only if $\theta$ is of real type, and 
	$1_{\fh_\C}$ occurs in ${\textstyle\bigwedge^2}(\theta)$ if and only if $\theta$ is of quaternionic type. 
\end{remark}

\begin{remark}[Tensor formulas]\label{rem:KoikeTerada}
Assume $\fl_\C$ is simple of type $B_n$, $C_n$ or $D_n$, with $n\geq3$. 
For non-negative integers $s\leq r$ and $q\leq p\leq n-2$, one has that (see e.g.\ \cite[p.~509--510]{KoikeTerada87} or \cite[\S1.64--1.66]{DictionaryOnLieAlgebras})
\begin{align*}
	\pi_{r\omega_1}\otimes \pi_{s\omega_1}&= \bigoplus_{j=0}^{s}\bigoplus_{i=0}^{j} \pi_{(r+s-2j)\omega_1+(j-i)\omega_2},\\
	\pi_{r\omega_1}\otimes \pi_{\omega_p}&= \pi_{r\omega_1+\omega_p}\oplus \pi_{(r-1)\omega_1+\omega_{p+1}} \oplus \pi_{(r-2)\omega_1+\omega_{p}}\oplus \pi_{(r-1)\omega_1+\omega_{p-1}},\\
	\pi_{\omega_p}\otimes \pi_{\omega_q}&= \bigoplus_{j=0}^{q}\bigoplus_{i=0}^{j} \pi_{\omega_{p+q-j-i}+\omega_{j-i}}.
\end{align*}
Here, $\omega_k=0$ if $k>n$. 
The assumption $p\leq n-2$ is not essential. 
Formulas also hold for $p=n-1$ for type $B_n$ and $p=n-1,n$ for type $C_n$.
Moreover, they hold for $p=n$ for type $B_n$ replacing $\omega_n$ by $\ee_1+\dots+\ee_n=2\omega_n$, and for $p=n-1,n$ for type $D_n$ replacing $\omega_{n-1}$ and $\omega_n$ by $\ee_1+\dots+\ee_{n-1}=\omega_{n-1}+\omega_n$ and $\ee_1+\dots+\ee_n=2\omega_n$ respectively. 

When $\fl_\C$ is of type $A_n$, similar identities are provided by Pieri's formula. 
We will directly cite \cite[Prop.~15.25]{FultonHarris-book} each time we use it to avoid stating it here. 
\end{remark}

\begin{remark}[Exterior and symmetric product]
\label{rem:ExtSym(suma)o(tensor)}
If $\theta_1,\theta_2$ are finite-dimensional representations of a complex semisimple Lie algebra $\fl_\C$, then we have 
\begin{equation*}
\label{eq2:Lambda^2(V+W)-Sym^2(V+W)}
\begin{aligned}
{\textstyle \bigwedge^2}(\theta_1\oplus \theta_2) &
\simeq {\textstyle \bigwedge^2}(\theta_1)
\oplus {\textstyle \bigwedge^2}(\theta_2)
\oplus \theta_1\otimes \,\theta_2
,
\\
\Sym^2(\theta_1\oplus \theta_2) &
\simeq \Sym^2(\theta_1)
\oplus \Sym^2(\theta_2)
\oplus \theta_1\otimes \,\theta_2
.
\end{aligned}
\end{equation*}

Now, if $\tau_1,\tau_2$ are finite-dimensional representations of arbitrary complex semisimple Lie algebras $(\fl_1)_\C,(\fl_2)_\C$ respectively, then 
\begin{equation*}
\label{eq2:Lambda^2(VxW)-Sym^2(VxW)}
\begin{aligned}
{\textstyle \bigwedge^2}(\tau_1\widehat\otimes \, \tau_2) &
\simeq {\textstyle \bigwedge^2}(\tau_1) \widehat\otimes \, \Sym^2(\tau_2)
\oplus \Sym^2(\tau_1) \widehat\otimes \, {\textstyle \bigwedge^2}(\tau_2)
,
\\
\Sym^2(\tau_1\widehat\otimes \, \tau_2) &
\simeq \Sym^2(\tau_1) \widehat\otimes \, \Sym^2(\tau_2)
\oplus {\textstyle \bigwedge^2}(\tau_1) \widehat\otimes \, {\textstyle \bigwedge^2}(\tau_2)
.
\end{aligned}
\end{equation*}
\end{remark}

\begin{table}
\caption{Families of non-symmetric strongly isotropy irreducible spaces.}\label{table:isotirred-families}
\renewcommand{\arraystretch}{1.3}
$
\begin{array}{ccccccccccc} 
\text{No.} & \fg & \fh & \rho &\text{Cond.} &
\Lambda_1 & \Lambda_2 
\\ \hline \hline 
\textup{I} & 
\su\big(\frac{n(n-1)}{2}\big) & 
\su(n) & 
\sigma_{\eta_2} &
	\text{\tiny 
	\begin{tabular}{c}
		$n\geq 5$ \\[-3pt]
		$d=\frac{n}{2}$, $n$ even\\[-3pt]
		$d=n$, $n$ odd
	\end{tabular}
	}
& \underset{(1)}{2\omega_1+2\omega_{N-1}}
& \omega_3+\omega_{N-3}
\\ \hline
\textup{II} & 
\su\big(\frac{n(n+1)}{2}\big) & 
\su(n) & 
\sigma_{2\eta_1}  &
	\text{\tiny 
	\begin{tabular}{c}
		$n\geq 3$ \\[-3pt]
		$d=\frac{n}{2}$, $n$ even\\[-3pt]
		$d=n$, $n$ odd
	\end{tabular}
	}
& \underset{(2)}{2\omega_1+2\omega_{N-1}}
& \omega_3+\omega_{N-3}
\\ \hline 
\textup{III}	& 
\su(pq) &
\su(p)\oplus\su(q) & 
\sigma_{\eta_1}\widehat\otimes \,\, \sigma_{\eta_1'}' & 
	\textup{\tiny $\begin{array}{c} 
		2\leq p\leq q \\[-3pt] 
		p+q\neq 4 \\[-3pt] 
		d=\op{lcm}(p,q)
	\end{array}$ } 
& \omega_2+\omega_{pq-2}
& 2\omega_1+2\omega_{pq-1}
\\ \hline 
\textup{IV} & 
\spp(n) &
\spp(1)\oplus\so(n) & 
\sigma_{\eta_1}\widehat\otimes \,\, \sigma_{\eta_1'}' &
n\geq 3 
& 2\omega_2\ 
& 4\omega_1\ 
\\ \hline 
\textup{V} & 
\so(n^2-1) &
\su(n) & 
\sigma_{\eta_1+\eta_{n-1}} &
n\geq 3 
& \ \ \, \omega_3 {}^{(3)}
& \omega_7
\\ \hline 
\textup{VI} & 
\so(2n^2-n-1) &
\spp(n) & 
\sigma_{\eta_2}  & 
n\geq 3 
& 3\omega_1
& { 2\omega_2}
\\ \hline 
\textup{VII} & 
\so(2n^2+n) &
\spp(n) & 
\sigma_{2\eta_1} &
n\geq 3 
& \omega_3
& \omega_7
\\ \hline 
\textup{VIII} & 
\so(4n) & 
\spp(1)\times\spp(n) &
\sigma_{\eta_1}\widehat\otimes \,\, \sigma_{\eta_1'}' & 
n\geq 2 
& \ \ \, \omega_4 {}^{(4)}
& 2\omega_2
\\ \hline 
\textup{IX} & 
\so\big(\frac{n(n-1)}{2}\big) &
\so(n) &
\begin{array}{c}
\sigma_{2\eta_2} \\
\sigma_{\eta_2}
\end{array} &
\begin{array}{c}
n=5 \\
n\geq 7
\end{array}
& \omega_3
& \omega_7
\\ \hline 
\textup{X} & 
\so\big(\frac{(n-1)(n+2)}{2}\big) &
\so(n) &
\sigma_{2\eta_1} & 
n\geq 5 
& 3\omega_1
& { 2\omega_2} 
\\ \hline
\end{array}
$

\smallskip 
\small 

$N=\frac{n(n-1)}{2}$ in Family~I and 
$N=\frac{n(n+1)}{2}$ in Family~II. 
(1) $\Lambda_1=3\omega_1$ for $n=6$. 
(2) $\Lambda_1=3\omega_1$ for $n=3$. 
{(3) $\Lambda_1=\ee_1+\ee_2+\ee_3=\omega_3+\omega_4$ for $n=3$.
(4) $\Lambda_1=\ee_1+\ee_2+\ee_3+\ee_4=2\omega_4$ for $n=2$.}

%(3) $\Lambda_1=8\omega_1$ for $n=4$. 
%(4) $\Lambda_2=2\omega_2$ for $n=4$ in order $\Lambda_1,\Lambda_2$ are linearly independent. 

\medskip

\end{table}

\begin{table}
\caption{Isolated cases of non-symmetric strongly isotropy irreducible spaces.}\label{table:isotirred-isolated}
\renewcommand{\arraystretch}{1.2}
\begin{minipage}[t]{0.55\textwidth}
$
\begin{array}[t]{c c c c c c } 
\text{No.} & \fg & \fh & \rho & \Lambda_1 & \Lambda_2 
\\ [0.5ex]
\hline 
\hline 
 1&\su(16) & \so(10)  & \sigma_{\eta_4} & 2\omega_2 & 2\omega_4 
\\  \hline 
 2&\su(27) & \fe_6    & \sigma_{\eta_1} & 3\omega_1 & 2\omega_3 
\\ \hline
 3&\so(7)  & \fg_2  & \sigma_{\eta_1} & \omega_3 & 
\\ \hline
 4& \so(133) & \fe_7 & \sigma_{\eta_3}  & \omega_3 & \omega_{11} 
\\ \hline
 5&\spp(2)  & \so(3)  & \sigma_{3\eta_1} & 4\omega_1 & 3\omega_2 
\\ \hline
 6&\spp(7)  & \spp(3) & \sigma_{\eta_3}  & 4\omega_1 & 2\omega_4 
\\ \hline
 7&\spp(10) & \su(6)  & \sigma_{\eta_3}  & 4\omega_1 & 2\omega_3 
\\ \hline
 8&\spp(16) & \so(12) & \sigma_{\eta_5} &  4\omega_1 & 2\omega_4 
\\ \hline
 9&\spp(28) & \fe_7  &  \sigma_{\eta_7} &  4\omega_1 & 3\omega_2 
\\ \hline
10&\so(14) & \fg_2  &  \sigma_{\eta_2} & \omega_3 & 2\omega_{2} 
\\ \hline
11& \so(16) & \so(9) &  \sigma_{\eta_4} &  2\omega_2 & 2\omega_1+\omega_4 
\\ \hline
12&\so(26) & \ff_4  &  \sigma_{\eta_4} & 3\omega_1 & 2\omega_2 
\\ \hline
13&\so(42) & \spp(4)& \sigma_{\eta_4}  & 2\omega_2 & 2\omega_1+\omega_4 
\\ \hline
14&\so(52) & \ff_4&  \sigma_{\eta_1}  & \omega_3 & \omega_{11} 
\\ \hline
15&\so(70) & \su(8) & \sigma_{\eta_4}  & 2\omega_2 & \omega_6 
\\ \hline
16& \so(78) & \fe_6 & \sigma_{\eta_2}   & \omega_3 & \omega_{9} 
\\ \hline
17&\so(128) & \so(16) & \sigma_{\eta_7} & 2\omega_2 & 2\omega_1+\omega_4 
\\ \hline
18&\so(248) & \fe_8 &  \sigma_{\eta_8}  & \omega_3 & \omega_{15} 
\\ \hline 
\end{array}
$
\end{minipage}
\begin{minipage}[t]{0.44\textwidth}
$
\begin{array}[t]{ccccc} 
\text{No.} & \fg & \fh %& \rho 
& \Lambda_1 & \Lambda_2 
\\ 
 \hline 
 \hline 
19& \fe_6 & \su(3) & %\sigma_{2(\eta_1+\eta_2)} &
	2\omega_1 & 3\omega_2 
\\  \hline 
20& \fe_6 & \su(3) & %\sigma_{\eta_1}\otimes \,\sigma_{\eta_1'}'\otimes \,\sigma_0''&
	\omega_1+\omega_6 & 3\omega_1 
\\ \hline 
21& \fe_6 & \fg_2 & %\sigma_{2\eta_1}& 
	2\omega_1 & 2\omega_3 
\\ \hline  
22& \fe_6 & \su(3)\oplus\fg_2 & 
%\begin{array}{l}
%	\sigma_{\eta_1}\otimes \,\sigma_{\eta_1'}'\oplus\\[-6pt]
%	\sigma_{2\eta_2}\otimes \,\sigma_0'
%\end{array}
%& 
	\omega_1+\omega_6 & 3\omega_1 
\\ \hline 
23& \fe_7 & \su(3) & %\sigma_{6\eta_1}\otimes \,\sigma_{6\eta_2'}'& 
	2\omega_7 & 3\omega_2 
\\ \hline 
24& \fe_7 & \su(6)\oplus\su(3)& 	%\sigma_{\eta_1}\otimes \,\sigma_{\eta_1'}'& 
	\omega_6 & 2\omega_1 
\\  \hline 
25& \fe_7 & \fg_2\oplus\spp(3) & %\sigma_{\eta_1}\otimes \,\sigma_{\eta_1'}'& 
	\omega_6 & 2\omega_1 
\\ \hline 
26& \fe_7 & \so(3)\oplus\ff_4 & %\sigma_{\eta_1}\otimes \,\sigma_{\eta_1'}'& 
	\omega_6 & 2\omega_1 
\\ \hline 
27& \fe_8 & \su(9) & %\sigma_{\eta_3}\oplus \sigma_{\eta_6}'\oplus \sigma_{\eta_1+\eta_8}''& 
	\omega_7 & \omega_6 
\\ \hline 
28& \fe_8 & \fe_6\oplus \su(3) &
	%\text{(*)}	& 
	\omega_1 & \omega_7 
\\ \hline 
29& \fe_8 & \fg_2\oplus\ff_4 & %?? & 
	\omega_1 & \omega_6 
\\ \hline 
30& \ff_4 & \su(3)\oplus\su(3) & %?? & 
	\omega_3 & 2\omega_2 
\\ \hline 
31& \ff_4 & \so(3)\oplus\fg_2 & %?? & 
	2\omega_3 & 2\omega_4 
\\ \hline 
32& \fg_2 & \so(3) & %?? & 
	2\omega_2 & 4\omega_1 
\\ \hline 
33& \fg_2 & \su(3) & {\omega_1} 
\\  \hline 
\end{array}
$

\medskip 

The embeddings $\rho:\fh_\C\hookrightarrow\fg_\C$ for the cases No.~19--33 are omitted; see \cite[Table~5]{LauretTolcachier-nu-stab}.
\end{minipage}
\end{table}

\section{The new Laplace eigenvalue}\label{sec:Lambda_2}

Let $G/H$ be a compact non-symmetric simply connected strongly isotropy irreducible space, where $G$ is a compact connected Lie group whose action on $G/H$ is almost effective.  
According to their classification, 
the Lie algebras $\fg,\fh$ of $G,H$ respectively must appear as a member of one of the ten families in Table~\ref{table:isotirred-families}, or as one of the 33 isolated cases in Table~\ref{table:isotirred-isolated}. 
The former spaces are numerated as in \cite{Schwahn-Lichnerowicz}, \cite{LauretTolcachier-nu-stab},  \cite{LauretRossiTolcachier-Saloff-Coste}, 
and the embedding $H\hookrightarrow G$ is determined by the corresponding embedding $\rho:\fh_\C\to\fg_\C$ given in the tables.

It was shown in \cite{LauretRossiTolcachier-Saloff-Coste} that the smallest positive eigenvalue of the Laplace-Beltrami operator of $(G/H,g_{\st})$ is given by $\lambda_1(G/H,g_{\st})= \lambda^{\pi_{\Lambda_1}}$, where $\Lambda_1$ is the dominant weight in Tables~\ref{table:isotirred-families}--\ref{table:isotirred-isolated}. 

The main goal of this section is to prove the following result. 

\begin{theorem}\label{thm:Lambda_2}
Let $(\fg,\fh)$ be a pair of Lie algebras from Tables~\ref{table:isotirred-families}--\ref{table:isotirred-isolated}, excluding the Isolated cases No.~3 and No.~33. 
The dominant weight $\Lambda_2$ in these tables satisfies 
$$
\pi_{\Lambda_2}\in\widehat G_H. 
$$ 
Consequently, $\lambda^{\pi_{\Lambda_2}}$ is an eigenvalue of the Laplace-Beltrami operator of $(G/H,g_{\st})$ with multiplicity $\geq \dim V_{\pi_{\Lambda_2}}$. 
\end{theorem}

Although $\lambda^{\pi_{\Lambda_1}}$ is the first eigenvalue of the Laplace-Beltrami operator, the result is not claiming that $\lambda^{\pi_{\Lambda_2}}$ is the second eigenvalue (i.e.\  the smallest positive eigenvalue $\lambda_2(G/H,g_{\st})$ strictly larger  than $\lambda_1(G/H,g_{\st})$). 
However, this occurs in some cases. 
The authors are not interested on computing explicitely $\lambda_2(G/H,g_{\st})$ for these spaces.

\begin{remark}
Note that $\Lambda_2$ is not given for the Isolated cases No.~3 and 33 in the tables. 
The corresponding pairs $(\Spin(7),\op{G}_2)$ and $(\op{G}_2,\SU(3))$ are Gelfand (also called spherical), and their spherical representations are (see \cite{Kramer-sphericalpairs})
\begin{align*}
\widehat {\Spin(7)}_{\op{G}_2}
	&= \{\pi_{k\omega_3}:k\in\Z_{\geq0}\},
&
\widehat {{(\op{G}_2)}}_{\SU(3)}
	&= \{\pi_{k\omega_1}:k\in\Z_{\geq0}\}. 
\end{align*}
In particular, there are no two linearly independent elements in $\PP^+(G,H)$ for these two cases. 
\end{remark}

The rest of the section is devoted to prove Theorem~\ref{thm:Lambda_2}.
The group $G$ is semisimple because its action on $G/H$ is almost effective. 
Furthermore, $H$ is connected because $G$ is connected and $G/H$ is simply connected. 
In particular, 
\begin{equation}\label{eq-prueba:Hcxo}
V_{\pi}^H=V_{\pi}^{\fh_\C}
= \{v\in V_\pi: d\pi(X)\cdot v=0\quad\forall \, X\in\fh_\C\}. 
\end{equation}
This allows us to restrict our attention at the Lie algebra level.

It is important to note that, for a pair $(\fg,\fh)$ in the tables, there might be several presentations of the simply connected space $G/H$ as quotient of compact Lie groups.
Wolf in \cite{Wolf68-isot-irred} obtained the effective presentation $G_0/H_0$, that can be also found in \cite[Tables~3--5]{LauretTolcachier-nu-stab}.
One can easily check in all cases that the dominant weight $\Lambda_2\in\PP^+(\fg_\C)$ is $G_0$-integral, thus $\pi_{\Lambda_2}\in  {(\widehat G_0)}_{H_0}$. 
For any other almost effective presentation $G/H$, $G$ covers $G_0$, that is, there is a central subgroup $\Gamma$ of $G$ such that $G/\Gamma\simeq G_0$. 
It follows that $\Lambda_2\in\PP^+(G)$.

It remains to show that, for each pair $(\fg,\fh)$ in Tables~\ref{table:isotirred-families}--\ref{table:isotirred-isolated} excepting for the Isolated cases No.~3 and No.~33, the weight $\Lambda_2$ satisfies
\begin{equation}
V_{\pi}^{\fh_\C}\neq0. 
\end{equation}
This is done as a case by case procedure. 
We will use that $\dim V_{\pi_{\Lambda_2}}^{\fh_\C}$ is the number of times that the trivial representation $1_{\fg_\C}$ of $\fh_\C$ occurs in the decomposition in irreducible terms of the restriction $\pi_{\Lambda_2}|_{\fh_\C}$. 
The computations are very similar to those made in \cite[\S3]{LauretRossiTolcachier-Saloff-Coste}. 
As mentioned in \S\ref{sec:preliminaries}, we will also denote by $\pi_{\Lambda}$, for any $\Lambda\in\PP^+(\fg_\C)$, the irreducible representation of $\fg_\C$ with highest weight $\Lambda$.

\subsection{Families I and II}
Set $\fh=\su(n)$, $N_1=\frac{n(n-1)}{2}$, $N_2=\frac{n(n+1)}{2}$, and $\fg_i=\su(N_i)$ for each $i=1,2$. 
The irreducible complex representations $\rho_1:=\sigma_{\eta_2}$ and $\rho_2:=\sigma_{2\eta_1}$ of $\fh_\C$ are of complex type and have dimension $N_1$ and $N_2$ respectively, so the image of $\rho_i$ is contained into $\su(V_{\rho_i})\simeq (\fg_i)_\C$. 
We want to show that $V_{\pi_{\Lambda_2}}^{\fh_\C}\neq0$ in both cases, where $\Lambda_2=\omega_3+\omega_{N_i-3}$.

We have the following tensor decomposition of $(\fg_i)_\C$-modules (see \cite[Prop 15.25]{FultonHarris-book}): 
\begin{equation*}
	\pi_{\omega_3}\otimes\pi_{\omega_{N_i-3}}
	\simeq \pi_{\omega_3+\omega_{N_i-3}} 
	\oplus \pi_{\omega_2+\omega_{N_i-2}}
	\oplus \pi_{\omega_1+\omega_{N_i-1}}
	\oplus \pi_0. 
\end{equation*}
In the proof of \cite[Thm.~3.17]{LauretRossiTolcachier-Saloff-Coste}, it was established that $V_{\pi_{\omega_1+\omega_{N_i-1}}}^{\fh_\C} =0= V_{\omega_2+\omega_{N_i-2}}^{\fh_\C}$. 
Thus
\begin{equation}\label{eq-prueba:fliaI-II}
\dim V_{\omega_3+\omega_{N_i-3}}^{\fh_\C} 
= \dim V_{\pi_{\omega_3}\otimes\pi_{\omega_{N_i-3}}} ^{\fh_\C}-1
. 
\end{equation}
Using Remark~\ref{rem:ExtSym-standard} (in the second identity below) we obtain that
\begin{align*}
\big( \pi_{\omega_3}\otimes\pi_{\omega_{N_i-3}} \big) |_{\fh_\C}&
	= \pi_{\omega_3}|_{\fh_\C} \otimes \pi_{\omega_{N_i-3}}|_{\fh_\C}
\\ & 
	=  {\textstyle\bigwedge^3}(\pi_{\omega_1})|_{\fh_\C} \otimes {\textstyle\bigwedge^3}(\pi_{\omega_{1}}^*)|_{\fh_\C}
=  {\textstyle\bigwedge^3}(\pi_{\omega_1}|_{\fh_\C})| \otimes {\textstyle\bigwedge^3}(\pi_{\omega_{1}}^*|_{\fh_\C})
	\\ &  
	\simeq {\textstyle\bigwedge}^3(\rho_i) \otimes {\textstyle\bigwedge^3}(\rho_i^*)
	= 	{\textstyle\bigwedge}^3(\rho_i) \otimes {\textstyle\bigwedge^3}(\rho_i)^*
.
\end{align*}

We claim that it is sufficient to show that ${\textstyle\bigwedge}^3(\rho_i)$ is not irreducible.
Indeed, if it decomposes as a sum of $m\geq2$ irreducible terms, then $1_{\fh_\C}$ occurs at least $m$ times in ${\textstyle\bigwedge}^3(\rho_i) \otimes {\textstyle\bigwedge^3}(\rho_i)^*$ by Remark~\ref{rem:rhoxrho^*}, so $\dim V_{\pi_{\omega_3}\otimes\pi_{\omega_{N_i-3}}} ^{\fh_\C}\geq m$, and therefore $\dim V_{\omega_3+\omega_{N_i-3}}^{\fh_\C} \geq m-1>0$ by \eqref{eq-prueba:fliaI-II}, yielding the required assertion.

We start with $i=1$. 
It is well known that ${\textstyle\bigwedge}^3(\rho_1)= {\textstyle\bigwedge}^3(\sigma_{\eta_2})$ is an $\fh_\C$-submodule of $\sigma_{\eta_2}\otimes \sigma_{\eta_2}\otimes \sigma_{\eta_2}$. By \cite[Prop 15.25]{FultonHarris-book}, 
\begin{equation*}
\sigma_{\eta_2}\otimes \sigma_{\eta_2}\otimes \sigma_{\eta_2}
\simeq 
\begin{cases}
2\sigma_{\eta_1} \oplus 3\sigma_{\eta_2+\eta_4} \oplus \sigma_{2\eta_1+\eta_4}
	\oplus \sigma_{2\eta_3} \oplus 2\sigma_{\eta_1+\eta_2+\eta_3}
	\oplus \sigma_{3\eta_2}
	&\text{ if }n=5,
\\
\sigma_{0}\oplus 2\sigma_{\eta_1+\eta_5} \oplus 3\sigma_{\eta_2+\eta_4} \oplus \sigma_{2\eta_1+\eta_4}
	\oplus \sigma_{2\eta_3} \oplus 2\sigma_{\eta_1+\eta_2+\eta_3}
	\oplus \sigma_{3\eta_2}
	&\text{ if }n=6,
\\
\sigma_{\eta_6}\oplus 2\sigma_{\eta_1+\eta_5} \oplus 3\sigma_{\eta_2+\eta_4} \oplus \sigma_{2\eta_1+\eta_4}
	\oplus \sigma_{2\eta_3} \oplus 2\sigma_{\eta_1+\eta_2+\eta_3}
	\oplus \sigma_{3\eta_2}
	&\text{ if }n\geq 7.
\end{cases}
\end{equation*}
It is a simple matter to check, by using the Weyl dimension formula, that $\dim {\textstyle\bigwedge}^3(\rho_1)
= \dim {\textstyle\bigwedge}^3(\sigma_{\eta_2})= \binom{N_1}{3}$ does not coincide with the dimension of any of the irreducible terms above.
This forces that ${\textstyle\bigwedge}^3(\rho_1)$ is not irreducible. 

We now see the case $i=2$, which is very similar, so we omit some details. 
We have that ${\textstyle\bigwedge}^3(\rho_2)= {\textstyle\bigwedge}^3(\sigma_{2\eta_1})$ is an $\fh_\C$-submodule of 
\begin{equation*}
\sigma_{2\eta_1}\otimes \sigma_{2\eta_1}\otimes \sigma_{2\eta_1}
\simeq 
\sigma_{6\eta_1} \oplus 2\sigma_{4\eta_1+\eta_2} \oplus 3\sigma_{2\eta_1+2\eta_2} \oplus \sigma_{3\eta_1+\eta_3} \oplus \sigma_{3\eta_2} \oplus 2\sigma_{\eta_1+\eta_2+\eta_3} \oplus \sigma_{2\eta_3}, 
\end{equation*}
and $\dim {\textstyle\bigwedge}^3(\rho_2)
= \dim {\textstyle\bigwedge}^3(\sigma_{2\eta_1})= \binom{N_2}{3}$ does not coincide with the dimension of any of the irreducible terms above.
Hence, ${\textstyle\bigwedge}^3(\rho_1)$ is not irreducible.

\subsection{Family III}
Set $\fh=\su(p)\oplus\su(q)$, and $\rho=\sigma_{\eta_1} \widehat\otimes\, \widehat \sigma_{\eta_1'}'$. 
Here, $\sigma_{\eta_1}$ and $\sigma_{\eta_1'}'$ are the standard representations of the simple ideals $\su(p)$ and $\su(q)$ respectively. 
We have that $\rho$ is an irreducible complex representation of $\fh_\C$ of complex type, so its image in included in $\su(V_{\rho})\simeq \su(pq)$. 
Here $\Lambda_2=2\omega_1+2\omega_{pq-1}$. 

It follows from \cite[Prop 15.25]{FultonHarris-book} that 
$ 
\pi_{2\omega_1}\otimes\pi_{2\omega_{pq-1}}
	\simeq \pi_{2\omega_1+2\omega_{pq-1}}
	\oplus \pi_{\omega_1+\omega_{pq-1}}
	\oplus \pi_{0}.
$ 
In the proof of \cite[Thm.~3.20]{LauretRossiTolcachier-Saloff-Coste}, it was established that $V_{\pi_{\omega_1+\omega_{pq-1}}}^{\fh_\C} =0$. 
Thus
$ 
\dim V_{2\omega_1+2\omega_{pq-1}}^{\fh_\C} 
= \dim V_{\pi_{2\omega_1}\otimes\pi_{2\omega_{pq-1}}} ^{\fh_\C}-1
. 
$ 
Using Remark~\ref{rem:ExtSym-standard} (in the second identity below) we obtain that
\begin{align*}
\pi_{2\omega_1}\otimes\pi_{2\omega_{pq-1}}|_{\fh_\C} & 
= \pi_{2\omega_1}|_{\fh_\C} \otimes\pi_{2\omega_{pq-1}}|_{\fh_\C}
= {\Sym^2}(\pi_{\omega_1})|_{\fh_\C} \otimes {\Sym^2}(\pi_{\omega_{pq-1}})|_{\fh_\C}
\\ & 
= {\Sym^2}(\pi_{\omega_1}|_{\fh_\C}) \otimes {\Sym^2}(\pi_{\omega_{pq-1}}|_{\fh_\C})
\simeq {\Sym}^2(\rho) \otimes {\Sym}^2(\rho)^* 
.
\end{align*}
Similarly as in the previous case, it suffices to show that ${\Sym}^2(\rho)$ is not $\fh_\C$-irreducible.
By Remarks~\ref{rem:ExtSym-standard} and \ref{rem:ExtSym(suma)o(tensor)}, we get
\begin{align*}
{\Sym}^2(\rho) &
= 
{\Sym}^2(\sigma_{\eta_1}\widehat\otimes \,\, \sigma_{\eta_1'}')
\simeq 
{\Sym}^2(\sigma_{\eta_1})\widehat\otimes \Sym^2(\sigma_{\eta_1'}')
	\, \oplus \,  {\textstyle\bigwedge}^2(\sigma_{\eta_1})\widehat\otimes {\textstyle\bigwedge}^2(\sigma_{\eta_1'}')
\\ 	&
\simeq 
\sigma_{2\eta_1}  \widehat\otimes \,\,  \sigma_{2\eta_{1}'}'
\, \oplus \, 
\sigma_{\eta_2}\widehat\otimes \,\,  \sigma_{\eta_{2}'}'
\end{align*}
which is clearly not irreducible. 

\subsection{Family IV}
It was established in the proof of \cite[Thm.~3.11]{LauretRossiTolcachier-Saloff-Coste} that $\dim V_{4\omega_1}^{\fh_\C}=1$. 

\subsection{Family V}
Set
$\fh=\su(n)$ and $\fg=\so(n^2-1)$. 
The embedding of $\fh_\C$ into $\fg_\C$ is given by the adjoint representation of $\fh_\C$, which is of real type, $\ad:\fh_\C\to \fg_\C\subset \gl(\fh_\C)$. 

We have that $\dim V_{\pi_{\omega_7}}^{\fh_\C}>0$ follows from the procedure using the Poincaré polynomial explained in the proof of \cite[Thm.~3.1]{LauretRossiTolcachier-Saloff-Coste}.

\subsection{Family VI} 
Set $\fh=\spp(n)$ and $\fg=\so(2n^2-n-1)$.
The embedding is given by the irreducible representation $\sigma_{\eta_2}$ of $\fh_\C$ of real type, i.e.\  $\rho =\sigma_{\eta_2}: \fh_\C\to \fg_\C\subset\gl(V_{\sigma_{\eta_2}})$.
Here $\Lambda_2=2\omega_2$.

It is well known that $\Sym^2(\pi_{\omega_1}) \simeq\pi_{2\omega_1}\oplus\pi_0$ and ${\textstyle\bigwedge^2}(\pi_{\omega_1}) \simeq\pi_{\omega_2}$ (see Remark~\ref{rem:ExtSym-standard}). 
Restricting them to $\fh_\C$ we have that
\begin{align*}
\pi_{2\omega_1}|_{\fh_\C}\oplus\sigma_0&
	\simeq {\Sym^2}(\pi_{\omega_1}|_{\fh_\C})  
	= {\Sym}^2(\rho) 
	= {\Sym}^2(\sigma_{\eta_2}) 
	\simeq \sigma_0\oplus \sigma_{\eta_{2}}\oplus\sigma_{2\eta_2}\oplus \sigma_{\eta_{4}},
\\
\pi_{\omega_2}|_{\fh_\C}&
	\simeq {\textstyle\bigwedge^2}(\pi_{\omega_1}|_{\fh_\C}) 
	= {\textstyle\bigwedge}^2(\rho) 
	= {\textstyle\bigwedge}^2(\sigma_{\eta_2})
	\simeq \sigma_{2\eta_1}\oplus \sigma_{\eta_1+\eta_3}
.
\end{align*}
The last two steps in each row were established in \cite[Lem.~3.5]{LauretRossiTolcachier-Saloff-Coste}. 
It follows immediately from Remark~\ref{rem:rhoxrho^*} that $\dim V_{\pi_{\omega_2}\otimes\pi_{\omega_2}}^{\fh_\C}=2$ (because $\sigma_{2\eta_1}$ and $\sigma_{\eta_1+\eta_3}$ are self-dual)
and 
$\dim V_{\pi_{2\omega_1}\otimes\pi_{\omega_2}}^{\fh_\C}=0$.

On the other hand, we have by Remark~\ref{rem:KoikeTerada} that 
\begin{equation}\label{eq-prueba:so(N)_omega2xomega2_2omega1xomega2}
\begin{aligned}
\pi_{2\omega_1}\otimes\pi_{\omega_2} &
	\simeq \pi_{2\omega_1+\omega_2}\oplus \pi_{\omega_1+\omega_3} \oplus \pi_{2\omega_1}\oplus \pi_{\omega_2}
, \\
\pi_{\omega_2}\otimes\pi_{\omega_2}&
	\simeq \pi_{\omega_4} \oplus \pi_{2\omega_2} \oplus \pi_{\omega_1+\omega_3} \oplus \pi_{2\omega_1} \oplus \pi_{\omega_2} \oplus \pi_0
.
\end{aligned}
\end{equation}
The first row implies 
$ 
0
	= \dim V_{\pi_{2\omega_1}\otimes\pi_{\omega_2}}^{\fh_\C} 
	= \dim V_{\pi_{2\omega_1+\omega_2}}^{\fh_\C}
	+ \dim V_{\pi_{\omega_1+\omega_3}}^{\fh_\C}
	+ \dim V_{\pi_{2\omega_1}}^{\fh_\C}
	+ \dim V_{\pi_{\omega_2}}^{\fh_\C} 
,
$ 
so all the terms at the right vanish. 
The second row gives 
$
2=
\dim V_{\pi_{\omega_2}\otimes\pi_{\omega_2}}^{\fh_\C}
	= \dim V_{\pi_{\omega_4}}^{\fh_\C} 
	+ \dim V_{\pi_{2\omega_2}}^{\fh_\C} 
	+ 0
	+ 0
	+ 0
	+ 1
, 
$
concluding that 
\begin{equation}\label{eq:dim+dim=1}
\dim V_{\pi_{2\omega_2}}^H+\dim V_{\pi_{\omega_4}}^H=1.
\end{equation}
This ensures that $V_{\pi_{\Lambda_2}}^{\fh_\C} \neq 0$ for some $\Lambda_2 \in \{2\omega_2,\omega_4\}$.

%The next goal is to show that $\dim V_{\pi_{2\omega_2}}^{\fh_\C}=1$.

\begin{claim}
$\displaystyle \pi_{\omega_3}|_{\fh_\C}= \sigma_{2\eta_1} \oplus \sigma_{\eta_1+\eta_3} \oplus \sigma_{2\eta_3} \oplus \sigma_{2\eta_1+\eta_2} \oplus \sigma_{2\eta_1+\eta_4}. $
\end{claim}

\begin{proof}
\renewcommand{\qedsymbol}{$\blacksquare$}
By \eqref{eq-prueba:so(N)_omega2xomega2_2omega1xomega2}, we have that
	\begin{align*}
		\big(\pi_{\omega_3} \oplus \pi_{\omega_1+\omega_2} \oplus \pi_{\omega_1} \big) |_{\fh_\C}&
		\simeq \big( \pi_{\omega_2}\otimes\pi_{\omega_{1}}\big) |_{\fh_\C}
		= \pi_{\omega_2}|_{\fh_\C} \otimes \pi_{\omega_{1}}|_{\fh_\C}
		\\ & 
		\simeq \big(\sigma_{\eta_1+\eta_3} \oplus \sigma_{2\eta_1}\big) \otimes  \sigma_{\eta_2}
		\\& 
		\simeq  
		2\sigma_{\eta_2} \oplus \sigma_{\eta_4} \oplus 2\sigma_{2\eta_1} \oplus 3\sigma_{\eta_1+\eta_3} \oplus \sigma_{\eta_1+\eta_5}
		\oplus \sigma_{2\eta_2} \oplus \sigma_{\eta_2+\eta_4}
		\\
		&\quad \oplus \sigma_{2\eta_3} \oplus 2\sigma_{2\eta_1+\eta_2}
		\oplus \sigma_{\eta_1+\eta_2+\eta_3} \oplus \sigma_{2\eta_1+\eta_4},
		\\
		\big(\pi_{3\omega_1} \oplus \pi_{\omega_1+\omega_2}
		\oplus \pi_{\omega_1} \big)|_{\fh_\C} &
		=\big(\pi_{2\omega_1}\otimes\pi_{\omega_{1}}\big)|_{\fh_\C}
		=\pi_{2\omega_1}|_{\fh_\C} \otimes \pi_{\omega_{1}}|_{\fh_\C}
		\\ & 
		\simeq \big(\sigma_{\eta_4} \oplus \sigma_{2\eta_2} \oplus \sigma_{\eta_2}\big) \otimes  \sigma_{\eta_2}
		\\ & 
		\simeq  
		3\sigma_{\eta_2} \oplus 2\sigma_{\eta_4} \oplus \sigma_{2\eta_1}
		\oplus \sigma_{\eta_6}
		\oplus 3\sigma_{\eta_1+\eta_3} 
		\oplus \sigma_{\eta_1+\eta_5} 
		\oplus 2\sigma_{2\eta_2} \oplus 2\sigma_{\eta_2+\eta_4}
		\\
		&\quad \oplus \sigma_{3\eta_2} \oplus \sigma_{2\eta_1+\eta_2}
		\oplus \sigma_{\eta_1+\eta_2+\eta_3} \oplus \sigma_{0}.
	\end{align*}

One can check that $\sigma_{2\eta_1} \oplus  \sigma_{2\eta_3} \oplus \sigma_{2\eta_1+\eta_2} \oplus \sigma_{2\eta_1+\eta_4} $ appears in the first formula but not in the second. 
Therefore, it is necessarily a summand in the decomposition of $\pi_{\omega_3}|_{\fh_\C}$.
Using the Weyl dimension formula, one gets
$$
\dim V_{\pi_{\omega_3}}- \dim(V_{\sigma_{2\eta_1}} \oplus  V_{\sigma_{2\eta_3}} \oplus V_{\sigma_{2\eta_1+\eta_2}} \oplus V_{\sigma_{2\eta_1+\eta_4}})= \tfrac{1}{2}\, (n-2)(n+1)(2n-1)(2n+1).
$$
It remains to decide how to write this number as a non-negative combination of the dimensions of rest of the irreducible constituent of $\big(\pi_{\omega_3} \oplus \pi_{\omega_1+\omega_2} \oplus \pi_{\omega_1} \big) |_{\fh_\C}$, namely
\begin{align*}
\dim V_{\sigma_{\eta_2}} &= (n-1)(2n+1), &
\dim V_{\sigma_{\eta_4}} &= \tfrac{n(n-3)(4n^2-1)}{6}, \\
\dim V_{\sigma_{2\eta_1}} &= n(2n+1), &
\dim V_{\sigma_{\eta_1+\eta_3}} &= \tfrac{(n-2)(n+1)(4n^2-1)}{2}, \\
\dim V_{\sigma_{\eta_1+\eta_5}} &= \tfrac{n(n-4)(n+1)(2n-3)(4n^2-1)}{18}, &		
\dim V_{\sigma_{2\eta_2}} &= \tfrac{n(n-1)(2n-1)(2n+3)}{3}, \\
\dim V_{\sigma_{\eta_2+\eta_4}} &= \tfrac{n^2(n-3)(4n^2-9)(2n+1)}{10}, &
%\dim V_{\sigma_{2\eta_3}} &= \tfrac{(n-2)(n^2-1)(4n^2-9)(2n+1)}{18}\, , &
\dim V_{\sigma_{2\eta_1+\eta_2}} &= \tfrac{n(n-1)(2n+1)(2n+3)}{2}, \\
\dim V_{\sigma_{\eta_1+\eta_2+\eta_3}} &= \tfrac{64n^2(n^2-4)(n^2-1)}{45}. &
%\dim V_{\sigma_{2\eta_1+\eta_4}} &= \tfrac{(n-3)(n^2-1)(4n^3-1)(2n+3)}{9}. & &
\end{align*}
One can check that the only way to do that is $\dim V_{\sigma_{\eta_1+\eta_3}}$, and the claim follows. 
\end{proof}

Since $\pi_{\omega_3}|_{\fh_\C}$ does not contain 
$\sigma_{\eta_2}\simeq \sigma_{\eta_2}^*$, 
Remark~\ref{rem:rhoxrho^*} yields $\sigma_0$ does not appear in $\pi_{\omega_3}\otimes\pi_{\omega_1}|_{\fh_\C}$. 
Since $\pi_{\omega_3}\otimes\pi_{\omega_1}\simeq \pi_{\omega_4}\oplus\pi_{\omega_3+\omega_1}\oplus \pi_{\omega_2} $ by Remark~\ref{rem:KoikeTerada}, we conclude that $\dim V_{\pi_{\omega_4}}^{\fh_\C}=0$. 
We conclude that $\dim V_{\pi_{2\omega_2}}^{\fh_\C}=1$ by \eqref{eq:dim+dim=1}, as claimed.

\subsection{Family VII}
Since $\ad:\fh_\C\to\fg_\C\subset\gl(\fh_\C)$ is the embedding, $V_{\pi_{\omega_7}}^{\fh_\C}\neq0$ as in Family V. 

\subsection{Family VIII}
It was established in the proof of \cite[Thm.~3.14]{LauretRossiTolcachier-Saloff-Coste} that $\dim V_{2\omega_2}^{\fh_\C}=1$. 

\subsection{Family IX}
Since $\ad:\fh_\C\to\fg_\C\subset\gl(\fh_\C)$ is the embedding, $V_{\pi_{\omega_7}}^{\fh_\C}\neq0$ as in Family V.

\subsection{Family X}
Set 
$\fh=\so(n)$, $\fg=\so(\frac{(n-1)(n+2)}{2})$,  $\rho=\sigma_{2\eta_1}:\fh_\C\to\fg_\C\subset\gl(V_{\sigma_{2\eta_1}})$, $\Lambda_2=2\omega_2$. 
This case is very similar to Family VI, so we will omit several details.  

The same argument as in Family VI and \cite[Lem.~3.8]{LauretRossiTolcachier-Saloff-Coste} gives
\begin{align*}
\pi_{2\omega_1}|_{\fh_\C}\oplus\sigma_0&
	\simeq {\Sym^2}(\pi_{\omega_1}|_{\fh_\C})  
	= {\Sym}^2(\rho) 
	= {\Sym}^2(\sigma_{2\eta_1}) 
	\simeq \sigma_0\oplus\sigma_{2\eta_1}\oplus \sigma_{4\eta_{1}}\oplus \sigma_{2\eta_{2}},
\\
\pi_{\omega_2}|_{\fh_\C}&
	\simeq {\textstyle\bigwedge^2}(\pi_{\omega_1}|_{\fh_\C}) 
	= {\textstyle\bigwedge}^2(\rho) 
	= {\textstyle\bigwedge}^2(\sigma_{2\eta_1})
	\simeq \sigma_{\eta_2}\oplus \sigma_{2\eta_1+\eta_2}
.
\end{align*}
Remark~\ref{rem:rhoxrho^*} gives 
$\dim V_{\pi_{\omega_2}\otimes\pi_{\omega_2}}^{\fh_\C}=2$ 
and 
$\dim V_{\pi_{2\omega_1}\otimes \pi_{\omega_2}}^{\fh_\C}=0$. 
Now, \eqref{eq-prueba:so(N)_omega2xomega2_2omega1xomega2} gives 
\begin{equation*}
\begin{aligned}
0
=\dim V_{\pi_{2\omega_1}\otimes\pi_{\omega_2}}^{\fh_\C} &
	= \dim V_{\pi_{2\omega_1+\omega_2}}^{\fh_\C}
	+ \dim V_{\pi_{\omega_1+\omega_3}}^{\fh_\C} 
	+ \dim V_{\pi_{2\omega_1}}^{\fh_\C}
	+ \dim V_{\pi_{\omega_2}}^{\fh_\C}
, \\
2
= \dim V_{\pi_{\omega_2}\otimes\pi_{\omega_2}}^{\fh_\C}&
	= \dim V_{\pi_{\omega_4}}^{\fh_\C} 
	+ \dim V_{\pi_{2\omega_2}}^{\fh_\C} 
	+ \dim V_{\pi_{\omega_1+\omega_3}}^{\fh_\C} 
	+ \dim V_{\pi_{2\omega_1}}^{\fh_\C} 
	+ \dim V_{\pi_{\omega_2}}^{\fh_\C} 
	+ \dim V_{\pi_0}^{\fh_\C}
\\ &
	= \dim V_{\pi_{\omega_4}}^{\fh_\C} 
	+ \dim V_{\pi_{2\omega_2}}^{\fh_\C} 
	+ 0
	+ 0
	+ 0
	+ 1
.
\end{aligned}
\end{equation*}
We conclude that $\dim V_{\pi_{\omega_4}}^{\fh_\C} 	+ \dim V_{\pi_{2\omega_2}}^{\fh_\C} =1$. 

One can prove that $\dim V_{\pi_{\omega_4}}^{\fh_\C} =0$, and therefore $\dim V_{\pi_{2\omega_2}}^{\fh_\C} =1$, using the same tedious argument as in Family VI.

\subsection{Isolated cases No.~1, 5--7, 10--12, and 19--32}
These cases can be easily done computing restriction of irreducible representations using \cite{Sage}. 
For instance, the following code decomposes  $\pi_{2\omega_4}|_{\fh_\C}$ in $\fh_\C$-irreducible terms for the Isolated  case No.~1, obtaining that $\sigma_0=1_{\fh_\C}$ occurs exactly once:

\medskip 

\begin{lstlisting}
	sage: G=WeylCharacterRing(['A',15], style='coroots')
	sage: H=WeylCharacterRing(['D',5], style='coroots')
	sage: b=branching_rule(G,"D5(1/2, 1/2, 1/2, 1/2, -1/2)", rule="plethysm")
	sage: G(0,0,0,2).branch(H,rule=b)
	 D5(0,0,0,0,0) + D5(0,0,0,1,1) + D5(2,0,0,0,0) + D5(1,0,0,2,0) 
	 + D5(1,0,0,0,2) + D5(0,2,0,0,0) + D5(0,1,0,1,1) + 2*D5(0,0,2,0,0) 
	 + D5(0,0,0,2,2) + D5(2,0,0,1,1) + D5(1,1,1,0,0) + D5(1,1,0,0,2) 
	 + D5(1,0,1,1,1) + D5(0,2,0,1,1) + D5(0,1,1,0,2) + D5(4,0,0,0,0) 
	 + D5(3,0,0,2,0) + D5(2,0,2,0,0) + D5(2,0,0,0,4) + D5(1,2,0,0,2) 
	 + D5(0,4,0,0,0)
\end{lstlisting}

\medskip 

Moreover, some of them appears in \cite{LieART}. 
For instance, for the Isolated case No.~5, the line $\textbf{30} = \textbf{1} +  \textbf{7} +  \textbf{9} +  \textbf{13}$ in \cite[p.~330]{LieART} (inside the table titled ``$\Sp(4)\to\SU(2)$ (S)'') yields $\pi_{3\omega_2}|_{\fh_\C}\simeq \sigma_{0} \oplus\sigma_{6\ee_1} \oplus\sigma_{8\ee_1} \oplus\sigma_{12\ee_1}$, obtaining that $\dim V_{\pi_{\Lambda_2}}^{\fh_\C}=1$.

Unfortunately, the computations of the restriction of irreducible representations using \cite{Sage} is not (currently) feasible for the rest of the isolated cases due to the high rank of $\fg_\C$.

\subsection{Isolated case No.~2}
Set
$\fh=\fe_6$, $\fg=\su(27)$,  $\rho=\sigma_{\eta_1}:\fh_\C\to\fg_\C\subset\gl(V_{\sigma_{\eta_1}})$, $\Lambda_2=2\omega_3$. 

Decomposing the tensor products using \cite[Prop 15.25]{FultonHarris-book}, one can check that
$
\pi_{\omega_3}\otimes\pi_{\omega_{3}} 
	\simeq \pi_{\omega_2}\otimes\pi_{\omega_{4}}
	\oplus \pi_{2\omega_3}
$, so 
$
\dim V_{\pi_{2\omega_3}}^{\fh_\C} 
= \dim V_{\pi_{\omega_3}\otimes\pi_{\omega_{3}}}^{\fh_\C} - \dim V_{\pi_{\omega_2}\otimes\pi_{\omega_{4}}}^{\fh_\C}
$.
Using Remark~\ref{rem:ExtSym-standard}, we get that
\begin{align*}
\pi_{\omega_2}\otimes\pi_{\omega_{4}}|_{\fh_\C}&
	= {\textstyle\bigwedge^2}(\pi_{\omega_1}|_{\fh_\C}) \otimes {\textstyle\bigwedge^4}(\pi_{\omega_{1}}|_{\fh_\C})  
	\simeq {\textstyle\bigwedge}^2(\sigma_{\eta_1}) \otimes {\textstyle\bigwedge^4}(\sigma_{\eta_{1}})
	\simeq \sigma_{\eta_3}\otimes \sigma_{\eta_2+\eta_5}, 
\\
\pi_{\omega_3}\otimes\pi_{\omega_{3}}|_{\fh_\C}&
	= {\textstyle\bigwedge^3}(\pi_{\omega_1}|_{\fh_\C}) \otimes {\textstyle\bigwedge^3}(\pi_{\omega_{1}}|_{\fh_\C})  
	\simeq {\textstyle\bigwedge}^3(\sigma_{\eta_1}) \otimes {\textstyle\bigwedge^3}(\sigma_{\eta_{1}})
	\simeq \sigma_{\eta_4}\otimes \sigma_{\eta_4}.
\end{align*}

The exterior products above can be computed for instance using \cite{Sage}. 
The following code returns all of them:
\medskip 

\begin{lstlisting}
	sage: H=WeylCharacterRing(['E',6], style='coroots')
	sage: H(1,0).exterior_power(2)
	 E6(0,0,1,0,0,0)	
	sage: H(1,0).exterior_power(3)
	 E6(0,0,0,1,0,0)
	sage: H(1,0).exterior_power(4)
	 E6(0,1,0,0,1,0)
\end{lstlisting}

\medskip

Now, since $\sigma_{\eta_3}^*\not\simeq\sigma_{\eta_2+\eta_5}$ and $\sigma_{\eta_4}^*\simeq\sigma_{\eta_4}$, 
we deduce that 
$\dim V_{\pi_{\omega_2}\otimes\pi_{\omega_{4}}}^{\fh_\C}= 0$ 
and
$\dim V_{\pi_{\omega_3}\otimes\pi_{\omega_{3}}}^{\fh_\C}= 1$
by Remark~\ref{rem:rhoxrho^*}, respectively. 
We conclude that $\dim V_{\pi_{2\omega_3}}^{\fh_\C}=1$.

\subsection{Isolated cases No.~4, 14, 16, and 18}
Since $\ad:\fh_\C\to\fg_\C\subset\gl(\fh_\C)$ is the embedding, one can use the same argument as in Family V.

\subsection{Isolated case No.~8}
Set $\fh=\so(12)$, $\fg=\spp(16)$,  $\rho=\sigma_{\eta_5}: \fh_\C\to\fg_\C\subset\gl(V_{\sigma_{\eta_5}})$, $\Lambda_2=2\omega_4$. 

On the one hand,  ${\textstyle\bigwedge^p}(\pi_{\omega_1})|_{\fh_\C}\simeq \pi_{\omega_p}|_{\fh_\C} 
\oplus {\textstyle\bigwedge^{p-2}}(\pi_{\omega_1})|_{\fh_\C}$ by Remark~\ref{rem:ExtSym-standard}, and one can check (e.g.\ using \cite{Sage}) that
\begin{align*}
{\textstyle\bigwedge^p}(\pi_{\omega_1})|_{\fh_\C}
\simeq {\textstyle\bigwedge^p}(\pi_{\omega_1}|_{\fh_\C})
= {\textstyle\bigwedge}^p(\sigma_{\eta_5}) 
\simeq
\begin{cases}
\sigma_{\eta_4}\oplus \sigma_0&p=2,\\
\sigma_{\eta_3+\eta_6}\oplus \sigma_{\eta_5}&p=3,\\
\sigma_{\eta_2+\eta_6}\oplus \sigma_{2\eta_3}\oplus \sigma_{\eta_4}\oplus \sigma_0&p=4,\\
\sigma_{\eta_2+\eta_3+\eta_6}\oplus \sigma_{\eta_1+3\eta_3}\oplus \sigma_{\eta_3+\eta_6}\oplus \sigma_{\eta_5}&p=5.
\end{cases}
\end{align*}
We thus obtain that 
$\pi_{\omega_5}|_{\fh_\C}=\sigma_{\eta_2+\eta_3+\eta_6}\oplus \sigma_{\eta_1+3\eta_3}$, 
$\pi_{\omega_4}|_{\fh_\C}=\sigma_{\eta_2+\eta_6}\oplus \sigma_{2\eta_3}$, and
$\pi_{\omega_3}|_{\fh_\C}=\sigma_{\eta_3+\eta_6}$.
Since every irreducible representation of $\fh_\C$ is self-dual, it follows from Remark~\ref{rem:rhoxrho^*} that 
\begin{align*}
\dim V_{\pi_{\omega_3}\otimes\pi_{\omega_3}}^{\fh_\C}&=1,& 
\dim V_{\pi_{\omega_4}\otimes\pi_{\omega_4}}^{\fh_\C}&=2,&
\dim V_{\pi_{\omega_3}\otimes\pi_{\omega_5}}^{\fh_\C}&=0.
\end{align*}

On the other hand, using Remark~\ref{rem:KoikeTerada}, one has that 
$
\pi_{\omega_3}\otimes\pi_{\omega_{3}}
\simeq  \pi_{2\omega_3} \oplus \pi_{2\omega_2} \oplus \pi_{2\omega_1} \oplus \pi_{0} \oplus \cdots
$. 
Taking into account that $V_{\pi_{\omega_3}\otimes\pi_{\omega_3}}^{\fh_\C}=1$, and $\pi_0$ occurs once in the decompositions of $\pi_{\omega_3}\otimes\pi_{\omega_3}$, it follows that $\sigma_0$ does not occur in the restriction to $\fh_\C$ of any other member of the decomposition, in particular
$\dim V_{\pi_{2\omega_3}}^{\fh_\C} =\dim V_{\pi_{2\omega_2}}^{\fh_\C} =\dim V_{\pi_{2\omega_1}}^{\fh_\C} =0$. 
Similarly, 
$
\pi_{\omega_4}\otimes\pi_{\omega_{4}}
\simeq (\pi_{\omega_5}\otimes\pi_{\omega_{3}}) \oplus \pi_{2\omega_4} \oplus \pi_{2\omega_3} \oplus \pi_{2\omega_2} \oplus \pi_{2\omega_1} \oplus \pi_{0}
$, thus
\begin{align*}
2
=\dim V_{\pi_{\omega_4}\otimes\pi_{\omega_{4}}}^{\fh_\C}
&
= \dim V_{\pi_{\omega_5}\otimes\pi_{\omega_{3}}}^{\fh_\C}
+ \dim V_{\pi_{2\omega_4}}^{\fh_\C} 
+ \dim V_{\pi_{2\omega_3}}^{\fh_\C} 
+ \dim V_{\pi_{2\omega_2}}^{\fh_\C} 
+ \dim V_{\pi_{2\omega_1}}^{\fh_\C} 
+ \dim V_{\pi_{0}}^{\fh_\C}
\\ & 
= 0
+ \dim V_{\pi_{2\omega_4}}^{\fh_\C} 
+ 0
+ 0
+ 0
+ 1.
\end{align*}
We conclude that 
$\dim V_{\pi_{2\omega_4}}^{\fh_\C} =1$, as asserted.

\subsection{Isolated case No.~9}
Set $\fh=\fe_7$, $\fg=\spp(28)$,  $\rho=\sigma_{\eta_7}:\fh_\C\to\fg_\C\subset\gl(V_{\sigma_{\eta_7}})$, $\Lambda_2=3\omega_2$.

It follows from Remark~\ref{rem:KoikeTerada} that 
\begin{align*}
	\pi_{\omega_2}\otimes\pi_{\omega_{4}} &
	\simeq \pi_{\omega_2} \oplus \pi_{\omega_4} \oplus \pi_{\omega_6}
	\oplus \pi_{\omega_1+\omega_3} \oplus \pi_{\omega_1+\omega_5} \oplus \pi_{\omega_2+\omega_4}, 
	\\
	\pi_{2\omega_1}\otimes\pi_{2\omega_{2}} &
	\simeq \pi_{2\omega_1+2\omega_2} 
	\oplus \pi_{\omega_1+\omega_2+\omega_3} \oplus \pi_{2\omega_1+\omega_2}
	\oplus \pi_{2\omega_3} \oplus \pi_{2\omega_2}
	\oplus \pi_{\omega_1+\omega_3}
	\oplus \pi_{2\omega_1},\\
	\pi_{2\omega_2}\otimes\pi_{\omega_{2}} &
	\simeq \pi_{3\omega_2} 
	\oplus \pi_{\omega_1+\omega_2+\omega_3} \oplus \pi_{2\omega_1+\omega_2}
	\oplus \pi_{\omega_2+\omega_4} \oplus \pi_{2\omega_2}
	\oplus \pi_{\omega_1+\omega_3}
	\oplus \pi_{\omega_2}.
\end{align*}
One can check (e.g.\ using \cite{Sage}) that 
$\pi_{\omega_2}|_{\fh_\C}\simeq\sigma_{\eta_6}$, $\pi_{\omega_4}|_{\fh_\C}\simeq\sigma_{\eta_4}$, 
$\pi_{2\omega_1}|_{\fh_\C}\simeq\sigma_{\eta_1}+\sigma_{2\eta_7}$, and
$\pi_{2\omega_2}|_{\fh_\C}\simeq\sigma_{2\eta_1}+\sigma_{\eta_6} +\sigma_{2\eta_6} +\sigma_{\eta_1+2\eta_7} +\sigma_{\eta_2+\eta_7}$. 
By Remark~\ref{rem:rhoxrho^*}, 
\begin{align*}
\dim V_{\pi_{2\omega_2}\otimes \pi_{\omega_2}}^{\fh_\C} &= 1,&
\dim V_{\pi_{\omega_2} \otimes \pi_{\omega_4}}^{\fh_\C} =
\dim V_{\pi_{2\omega_1} \otimes \pi_{2\omega_2}}^{\fh_\C}
&= 0.
\end{align*}
The identities at the right ensure that $\dim V_{\pi}^{\fh_\C}=0$ for every irreducible term $\pi$ occurring in $\pi_{\omega_2} \otimes \pi_{\omega_4}$ or $\pi_{2\omega_1} \otimes \pi_{2\omega_2}$ above. 
It follows that $1=\dim V_{\pi_{2\omega_2}\otimes \pi_{\omega_2}}^{\fh_\C} = \dim V_{\pi_{3\omega_2}}^{\fh_\C}$, as asserted. 

\subsection{Isolated case No.~13}
Set $\fh=\spp(4)$, $\fg=\so(42)$,  $\rho=\sigma_{\eta_4}:\fh_\C\to\fg_\C\subset\gl(V_{\sigma_{\eta_1}})$, $\Lambda_2=2\omega_1+\omega_4$.

By using Remark~\ref{rem:ExtSym-standard} and computing the corresponding symmetric and exterior power of $\sigma_{\eta_4}$ (e.g.\ using \cite{Sage}), one has that
\begin{align*}
(\pi_0 \oplus \pi_{2\omega_1}) \otimes \pi_{\omega_{4}}|_{\fh_\C}&
	\simeq  {\Sym^2}(\pi_{\omega_1}|_{\fh_\C}) \otimes {\textstyle\bigwedge^4}(\pi_{\omega_{1}}|_{\fh_\C})  
	= {\Sym^2}(\sigma_{\eta_4}) \otimes {\textstyle\bigwedge^4}(\sigma_{\eta_{4}}) \\ & 
	\simeq (\sigma_0 \oplus \sigma_{2\eta_2} \oplus \sigma_{2\eta_4})\otimes (\sigma_{\eta_2} \oplus \sigma_{2\eta_2} \oplus \sigma_{\eta_2+\eta_4} \oplus \cdots), 
\\
\pi_{\omega_1}\otimes\pi_{\omega_{3}}|_{\fh_\C}&
	\simeq\sigma_{\eta_4} \otimes {\textstyle\bigwedge^3}(\sigma_{\eta_{4}}) 
	\simeq \sigma_{\eta_4} \otimes (\sigma_{\eta_1+\eta_3} \oplus \sigma_{2\eta_1+\eta_4} \oplus \sigma_{\eta_2+2\eta_3}),
\\
\pi_{\omega_1}\otimes\pi_{\omega_{5}}|_{\fh_\C}&
	\simeq\sigma_{\eta_4} \otimes {\textstyle\bigwedge^5}(\sigma_{\eta_{4}})  
	\simeq \sigma_{\eta_4} \otimes (\sigma_{\eta_2} \oplus \sigma_{\eta_1+\eta_3} \oplus \sigma_{2\eta_2} \oplus \cdots)
. 
\end{align*}
Since all irreducible representations of $\fh_\C$ are self-dual, it follows from Remark~\ref{rem:rhoxrho^*} that 
\begin{align*}
\dim V_{(\pi_0 \oplus \pi_{2\omega_1}) \otimes \pi_{\omega_{4}}}^{\fh_\C}&=1,&
\dim V_{\pi_{\omega_1}\otimes\pi_{\omega_{3}}}^{\fh_\C}
= \dim V_{\pi_{\omega_1}\otimes\pi_{\omega_{5}}}^{\fh_\C}
&=0.
\end{align*}

On the other hand we have from Remark~\ref{rem:KoikeTerada} that 
\begin{equation}\label{eq-prueba:case13}
\begin{aligned}
	\pi_{2\omega_1}\otimes\pi_{\omega_{4}} &
	\simeq \pi_{\omega_4} 
	\oplus \pi_{\omega_1+\omega_3}
	\oplus \pi_{\omega_1+\omega_5}
	\oplus \pi_{2\omega_1+\omega_4}, \\
	\pi_{\omega_1}\otimes\pi_{\omega_{3}} &
	\simeq \pi_{\omega_2} \oplus \pi_{\omega_4}
	\oplus \pi_{\omega_1+\omega_3}, \\
	\pi_{\omega_1}\otimes\pi_{\omega_{5}} &
	\simeq \pi_{\omega_4} \oplus \pi_{\omega_6}
	\oplus \pi_{\omega_1+\omega_5}.
\end{aligned}
\end{equation}
Since $\sigma_0$ does not occur in the restriction to $\fh_\C$ in the second and third row, we deduce that $\dim V_{\pi}^{\fh_\C}=0$ for each irreducible constituent $\pi$. 
We conclude that 
\begin{align*}
1&
=
\dim V_{(\pi_0 \oplus \pi_{2\omega_1}) \otimes \pi_{\omega_{4}}}^{\fh_\C}
= 
\dim V_{\pi_{\omega_{4}}}^{\fh_\C}
+
\dim V_{\pi_{2\omega_1} \otimes \pi_{\omega_{4}}}^{\fh_\C}
\\ &
= 0
+ \dim V_{\pi_{\omega_{4}}}^{\fh_\C}
+ \dim V_{\pi_{\omega_{1}+\omega_3}}^{\fh_\C}
+ \dim V_{\pi_{\omega_{1}+\omega_5}}^{\fh_\C}
+ \dim V_{\pi_{2\omega_{1}+\omega_4}}^{\fh_\C}
= \dim V_{\pi_{2\omega_{1}+\omega_4}}^{\fh_\C}.
\end{align*}

\subsection{Isolated case No.~15}
Set $\fh=\su(8)$, $\fg=\so(70)$,  $\rho=\sigma_{\eta_4}:\fh_\C\to\fg_\C\subset \gl(V_{\sigma_{\eta_4}})$, $\Lambda_2=\omega_6$. 

One has that 
$
\pi_{\omega_6}|_{\fh_\C}
\simeq {\textstyle\bigwedge}^6(\pi_{\omega_1}|_{\fh_\C})
= {\textstyle\bigwedge}^6(\sigma_{\eta_4})
$.
One can easily check (e.g.\ using \cite{Sage}) that its decomposition as $\fh_\C$-irreducible modules is of the form $\sigma_{0} \oplus \sigma_{\eta_1+\eta_7}\oplus \dots$, so $\dim V_{\pi_{\omega_6}}^{\fh_\C}=1$.

\subsection{Isolated case No.~17}
Set $\fh=\so(16)$, $\fg=\so(128)$,  $\rho=\sigma_{\eta_7}:\fh_\C\to\fg_\C\subset\gl(V_{\sigma_{\eta_1}})$, $\Lambda_2=2\omega_1+\omega_4$.

This case is very similar to the Isolated case No.~13. We have that 
\begin{align*}
	(\pi_0 \oplus \pi_{2\omega_1}) \otimes \pi_{\omega_{4}}|_{\fh_\C}&
	= {\Sym^2}(\sigma_{\eta_7}) \otimes {\textstyle\bigwedge^4}(\sigma_{\eta_{7}}) 
	\simeq (\sigma_0 \oplus \sigma_{\eta_4} \oplus \sigma_{2\eta_7})\otimes (\sigma_{\eta_4} \oplus \sigma_{2\eta_8} \oplus \sigma_{2\eta_1} \oplus \cdots), 
	\\
	\pi_{\omega_1}\otimes\pi_{\omega_{3}}|_{\fh_\C}&
	\simeq\sigma_{\eta_7} \otimes {\textstyle\bigwedge^3}(\sigma_{\eta_{7}}) 
	\simeq \sigma_{\eta_7} \otimes (\sigma_{\eta_2+\eta_7} \oplus \sigma_{\eta_1+\eta_8} \oplus \sigma_{\eta_5+\eta_8}),
	\\
	\pi_{\omega_1}\otimes\pi_{\omega_{5}}|_{\fh_\C}&
	\simeq\sigma_{\eta_7} \otimes {\textstyle\bigwedge^5}(\sigma_{\eta_{7}})  
	\simeq \sigma_{\eta_7} \otimes (\sigma_{\eta_1+\eta_8} \oplus \sigma_{\eta_2+\eta_7} \oplus \sigma_{\eta_3+\eta_8} \oplus \cdots).
\end{align*}
Since all irreducible representations of $\fh_\C$ are self-dual, it follows from Remark~\ref{rem:rhoxrho^*} that 
\begin{align*}
	\dim V_{(\pi_0 \oplus \pi_{2\omega_1}) \otimes \pi_{\omega_{4}}}^{\fh_\C}&=1,&
	\dim V_{\pi_{\omega_1}\otimes\pi_{\omega_{3}}}^{\fh_\C}
	= \dim V_{\pi_{\omega_1}\otimes\pi_{\omega_{5}}}^{\fh_\C}
	&=0.
\end{align*}

The decompositions in \eqref{eq-prueba:case13} still hold in this case, so the rest of the proof is the same, and $\dim V_{\pi_{2\omega_{1}+\omega_4}}^{\fh_\C}=1$.

\smallskip

The proof of Theorem~\ref{thm:Lambda_2} is now complete.

\section{Proof of the main theorem}\label{sec:isotropyirredspaces}

In this section we prove Theorem~\ref{thm1:main}, which will be a direct consequence of Theorem~\ref{thm:standard-fails}.

The following Lie theoretical characterization was proved by Petrecca and Röser~\cite[Prop.~1.6]{PetreccaRoeser18}. 
This result is an extension to the homogeneous case of Schueth's study \cite{Schueth17} done for the compact Lie group case (i.e.\ $H=\{e\}$).

\begin{proposition}\label{prop:PetreccaRoser}
For an arbitrary compact homogeneous Riemannian manifold $(G/H,g)$, all the eigenspaces of its Laplace-Beltrami operator $\Delta$ are irreducible as $G$-modules if and only if the following three conditions are satisfied:
\begin{enumA}
\item \label{item(a)}
For $\pi_1,\pi_2\in\widehat G_H$, with $\pi_1\not\simeq\pi_2$ and $\pi_1\not\simeq\pi_2^*$, the linear maps $d\pi_1(C_g)|_{V_\pi^H}$ and $d\pi_2(C_g)|_{V_\pi^H}$ do not share any eigenvalue. 

\item Every eigenvalue of $d\pi(C_g)|_{V_\pi^H}$ has multiplicity one, for any $\pi\in\widehat G_H$ of real or complex type. 

\item Every eigenvalue of $d\pi(C_g)|_{V_\pi^H}$ has multiplicity two, for any $\pi\in\widehat G_H$ of quaternionic type. 
\end{enumA}
Here, $d\pi(C_g)=\sum_{j} d\pi(X_j)^2:V_\pi\to V_\pi$, where $\{X_j\}$ is any orthonormal basis of $(\fp,\innerdots_g)$, for any $\pi\in\widehat G$. 
\end{proposition}

To prove Theorem~\ref{thm1:main}, we will show that \eqref{item(a)} fails for all simple connected strongly isotropy irreducible spaces different to a CROSS and to $\op{G}_2/\SU(3)$, $\Spin(7)/\op{G}_2$. 
We next prove a more general result.

\begin{theorem}\label{thm:standard-fails}
Let $H$ be a closed subgroup of a compact semisimple Lie group $G$, and let $\Lambda_1,\Lambda_2$ be linearly independent elements in $\PP^+(G,H)$. 
%If $\Lambda_1,\Lambda_2\in\PP^+(G,H)$ (i.e.\ $\pi_{\Lambda_i}\in\widehat G_H$ for $i=1,2$) are linearly independent and satisfy
If 
\begin{equation}\label{eq:condition-duality}
\pi_{a_1\Lambda_1+a_2\Lambda_2}\not\simeq \pi_{b_1\Lambda_1+b_2\Lambda_2}^*
\qquad\forall\, (a_1,a_2)\neq(b_1,b_2)\text{ in } \Z_{\geq0}\times\Z_{\geq0},
\end{equation} 
then the Laplace-Beltrami operator associated to $(G/H,g_{\st})$ has at least one real eigenspace that is not irreducible as a $G$-module.  

Moreover, \eqref{eq:condition-duality} holds if $\sigma_{\Lambda_i}$ is self-dual for some $i=1,2$. 
\end{theorem}

\begin{proof}
We will show that condition \eqref{item(a)} in Proposition~\ref{prop:PetreccaRoser} does not hold for $g=g_{\st}$. We have that $d\pi(C_g)|_{V_\pi^H} =d\pi(\Cas_\fg)|_{V_\pi^H} = \lambda^\pi\,\Id_{V_\pi^H}$ because $\fh$ acts trivially on $V_\pi^H$. 
Therefore,  \eqref{item(a)} is equivalent to 
\begin{equation}\label{eq:(a)}
\lambda^{\pi_1}\neq\lambda^{\pi_2}
\qquad\forall \, \pi_1,\pi_2\in\widehat G_H: 
\pi_1\not\simeq\pi_2
\text{ and }\pi_1\not\simeq\pi_2^*. 
\end{equation}

We consider the polynomial in two variables given by 
\begin{equation*}
\begin{aligned}
P(x_1,x_2) &
= \kil_\fg^*(x_1\Lambda_1+x_2\Lambda_2, x_1\Lambda_1+x_2\Lambda_2+2\rho_\fg)
\\ & 
= \kil_\fg^*(\Lambda_1, \Lambda_1)\, x_1^2
+  2\kil_\fg^*(\Lambda_1,\Lambda_2)\, x_1x_2
+ \kil_\fg^*(\Lambda_2, \Lambda_2)\, x_2^2
\\ & \quad 
+ \kil_\fg^*(\Lambda_1, 2\rho_\fg)\, x_1
+ \kil_\fg^*(\Lambda_2, 2\rho_\fg)\, x_2
.
\end{aligned}
\end{equation*}
It is well known that there is $c\in\N$ such that $cP$ is integer-valued on $\Z\times\Z$. 
Furthermore, the discriminant of $P$, namely
\begin{equation}
\big(2\kil_\fg^*(\Lambda_1,\Lambda_2)\big)^2
-
4\kil_\fg^*(\Lambda_1, \Lambda_1)
\kil_\fg^*(\Lambda_2, \Lambda_2)
,
\end{equation}
is negative by Cauchy-Schwarz inequality since $-\kil_\fg^*(\cdot,\cdot)$ is an inner product and $\Lambda_1,\Lambda_2$ are assumed linearly independent.

Gjaldb{\ae}k \cite[Thm.~1]{Gjaldbaek} proved that any integer-valued quadratic polynomial with non-zero discriminant cannot be injective on the integer lattice points of any affine convex cone. 
It follows that $cP$ (and therefore $P$) is not injective on $\Z_{\geq0}\times\Z_{\geq0}$, that is, there are $(a_1,a_2)\neq (b_1,b_2)$ in $\Z_{\geq0}\times\Z_{\geq0}$ such that $P(a_1,a_2)=P(b_1,b_2)$. 

Since $\PP^+(G,H)$ is a monoid, $a_1\Lambda_1+a_2\Lambda_2$ and $b_1\Lambda_1+b_2\Lambda_2$ are in $\PP^+(G,H)$. 
We set 
$\pi_1:=\pi_{a_1\Lambda_1+a_2\Lambda_2}, \pi_2:=\pi_{b_1\Lambda_1+b_2\Lambda_2} \in \widehat G_H$. 
They are clearly not equivalent because their highest weights are different. 
Moreover, $\pi_1\not\simeq\pi_2^*$ because of \eqref{eq:condition-duality}. 
It follows from \eqref{eq:Casimir} that
\begin{equation*}
-\lambda^{\pi_{1}}
= P(a_1,a_2)
= P(b_1,b_2)
= -\lambda^{\pi_{2}}.
\end{equation*}
We conclude that \eqref{eq:(a)} fails. %, and the proof is complete. 

It remains to prove the second assertion. 
We suppose that $\pi_{\Lambda_1}^*\simeq \pi_{\Lambda_1}$ and $\pi_{a_1\Lambda_1+a_2\Lambda_2}\simeq \pi_{b_1\Lambda_1+b_2\Lambda_2}^*$ for some $(a_1,a_2), (b_1,b_2)\in \Z_{\geq0}\times\Z_{\geq0}$.
We want to show that $(a_1,a_2)=(b_1,b_2)$.

Let $w_0$ be the longest element of the Weyl group of $(\fg_\C,\ft_\C)$. % and set $\sigma=-w_0$. 
It is well known that 
$-w_0$ is a linear involution of $(\mi\ft)^*$,
$-w_0(\Phi^+(\fg_\C,\ft_\C)) = \Phi^+(\fg_\C,\ft_\C)$, and 
$\pi_{\Lambda}^* \simeq\pi_{-w_0(\Lambda)}$ for any $\Lambda\in\PP^+(\fg_\C)$ (see e.g.\ \cite[Thm.~3.2.13]{GoodmanWallach-book-Springer}). 

We have that $-w_0(\Lambda_1)=\Lambda_1$ and $a_1\Lambda_1+a_2\Lambda_2= -w_0(b_1\Lambda_1+b_2\Lambda_2)$. 
It follows that
\begin{equation}\label{eq:sigma(Lambda_1)}
\begin{aligned}
-b_2w_0(\Lambda_2)& =(a_1-b_1)\Lambda_1+a_2\Lambda_2,\\
a_2w_0(\Lambda_2)& 
=(a_1-b_1)\Lambda_1-b_2\Lambda_2. 
\end{aligned}
\end{equation}
The second identity follows from the first one after applying $-w_0$. 
If $a_2=0$ or $b_2=0$, \eqref{eq:sigma(Lambda_1)} implies that $a_1=b_1$ and $a_2=b_2=0$ because $\Lambda_1,\Lambda_2$ are linearly independent. 
If $a_2$ and $b_2$ are both positive, adding $a_2$ times the first row in \eqref{eq:sigma(Lambda_1)} to $b_2$ times the second one yields
\begin{equation*}
0= (a_2+b_2)(a_1-b_1)\Lambda_1 + (a_2^2-b_2^2)\Lambda_2.
\end{equation*}
Again, the linear independence of $\Lambda_1,\Lambda_2$ forces $(a_1,a_2)=(b_1,b_2)$, and the proof is complete.
\end{proof}

\begin{remark}
It is not a difficult task to modify the proof of Theorem~\ref{thm:standard-fails} to conclude that the Laplace-Beltrami operator has indeed infinitely many $G$-reducible real eigenspaces. 
\end{remark}

\begin{proof}[Proof of Theorem~\ref{thm1:main}]
We recall that 
	$G$ is a compact connected Lie group, 
	$H$ is a closed subgroup of $G$, 
	the action of $G$ on $G/H$ is almost effective, and $G/H$ is a simply connected strongly isotropy irreducible space.

According to the classification of simply connected compact strongly isotropy irreducible spaces, 
or either $(G,H)$ is an irreducible symmetric pair or the Lie algebras $\fg,\fh$ of $G,H$ respectively must appear in Tables~\ref{table:isotirred-families}--\ref{table:isotirred-isolated}.

We start assuming that $(G,H)$ corresponds to the isolated cases No.~3 or No.~33 in Table~\ref{table:isotirred-isolated}. 
The only presentation of the corresponding simply connected spaces are $G/H$ with $(G,H)=(\op{G}_2,\SU(3)), (\Spin(7),\op{G}_2)$ respectively. 
As we mentioned in \S\ref{sec:intro}, \cite[Prop.~C.I.8]{BergerGauduchonMazet-LNM} proves that all eigenspaces of the Laplace-Beltrami operator associated to their standard metrics are $G$-irreducible. 

We now assume that $(G,H)$ is not a symmetric pair or either $(\op{G}_2,\SU(3))$ and $(\Spin(7),\op{G}_2)$. 
Tables~\ref{table:isotirred-families}--\ref{table:isotirred-isolated} show two linearly independent elements in  $\Lambda_1,\Lambda_2\in\PP^+(G,H)$. 
It is an easy matter to check in all cases that \eqref{eq:condition-duality} holds. 
Now, Theorem~\ref{thm:standard-fails} implies that the Laplacian of $(G/H,g_{\st})$ has at least one $G$-reducible eigenspace.

Again by \cite[Prop.~C.I.8]{BergerGauduchonMazet-LNM}, we know that if $(G,H)$ is a CROSS, then all eigenspaces are $G$-irreducible. 

It remains to prove that there is at least one $G$-reducible space if $(G,H)$ is an irreducible symmetric pair of real rank $r\geq2$. 
This is an alternative proof of \cite[Thm.~4]{PetreccaRoeser18}. 

Using the theory of restricted weights (see e.g.\ \cite[\S5]{Takeuchi}) one can describe explicitely the set $\PP^+(G,H)$ as a free monoid with generators $\mu_1,\dots,\mu_r$.  
See \cite[\S12.3.3]{GoodmanWallach-book-Springer} for their explicit values corresponding to the classical compact irreducible symmetric spaces (i.e.\ type AI, AII, AIII, BDI, CI, CII, DIII). 
For any $i\neq j$, we set $\Lambda_1=\mu_i$ and $\Lambda_2=\mu_i+\mu_j$, which are clearly linearly independent. 

By Theorem~\ref{thm:standard-fails}, it only remains to check that \eqref{eq:condition-duality} holds. 
Let $(a_1,a_2),(b_1,b_2)\in\Z_{\geq0}\times\Z_{\geq0}$. 
Clearly, $\pi_{a_1\Lambda_1+a_2\Lambda_2}\simeq \pi_{b_1\Lambda_1+b_2\Lambda_2}$ implies $(a_1,a_2)=(b_1,b_2)$. 
We now assume that $\pi_{a_1\Lambda_1+a_2\Lambda_2}\simeq \pi_{b_1\Lambda_1+b_2\Lambda_2}^*$. 
It is well known that there is a permutation $\tau$ of $\{1,\dots,r\}$ with $\tau^2=\op{id}$ such that $\pi_{\mu_i}^*\simeq \pi_{\mu_{\tau(i)}}$ for all $1\leq i\leq r$. 
Now, 
$\pi_{(a_1+a_2)\,\mu_i+a_2\,\mu_j}= \pi_{a_1\Lambda_1+a_2\Lambda_2}\simeq \pi_{b_1\Lambda_1+b_2\Lambda_2}^* \simeq\pi_{(b_1+b_2)\,\mu_{\tau(i)}+b_2\,\mu_{\tau(j)}}$ is equivalent to 
\begin{equation}\label{eq:coeficientes}
(a_1+a_2)\,\mu_i+a_2\,\mu_j
=(b_1+b_2)\,\mu_{\tau(i)}+b_2\,\mu_{\tau(j)}.
\end{equation}
We next show that \eqref{eq:coeficientes} forces $(a_1,a_2)=(b_1,b_2)$ for any $\tau$. 
\begin{itemize}
%\item 
%If $\tau(i)=i$ or $\tau(j)=j$, the second assertion in Theorem~\ref{thm:standard-fails} gives $(a_1,a_2)=(b_1,b_2)$.

\item 
If $\tau(i)=i$ and $\tau(j)=j$, \eqref{eq:coeficientes} gives $a_1+a_2=b_1+b_2$ and $a_2=b_2$, thus $(a_1,a_2)=(b_1,b_2)$.

\item 
If $\tau(i)=j$ and $\tau(j)=i$,  
\eqref{eq:coeficientes} gives $a_1+a_2=b_2$ and $a_2=b_1+b_2$. 
Their difference gives $a_1=-b_1$, which forces $a_1=b_1=0$ and $a_2=b_2$. 

\item 
If $\tau(i)=i$ and $\tau(j)\notin\{i,j\}$, 
\eqref{eq:coeficientes} gives $a_1+a_2=b_1+b_2$ and $a_2=b_2=0$, so $a_1=b_1$.

\item 
The case $\tau(i)\notin\{i,j\}$ and $\tau(j)=j$ is analogous to the previous one. 

\item 
If $\{\tau(i),\tau(j)\}\cap\{i,j\}=\emptyset$, 
\eqref{eq:coeficientes} gives  $(a_1,a_2)=(b_1,b_2)=(0,0)$.
\end{itemize}
Since $\tau$ is an involution, there are no further cases, so \eqref{eq:condition-duality} holds, and the proof is complete. 
\end{proof}

\begin{remark}
Instead of the Laplace-Beltrami operator acting on real-valued smooth functions (i.e.\ on $C^\infty(M,\R)$), one may work with the complex one $\Delta^\C$ acting on $C^\infty(M,\C)$. 
Representation theoretically speaking, this case is simpler because Peter and Weyl Theorem immediately implies for $(G/H,g_{\st})$ that 
$$
\op{Eig}_{\Delta^\C} (\lambda)\simeq  \bigoplus_{\pi\in\widehat G:\, \lambda^\pi=\lambda} (\dim V_\pi^H)\cdot \pi
$$
as complex $G$-modules. 
Moreover, a necessary condition to ensure the $G$-irreducibility of all eigenspaces of $\Delta^\C$ is that $\widehat G_H$ does not contain representations of complex or quaternionic type. 
Indeed, 
	if $\pi\in\widehat G_H$ is of quaternionic type, since $\dim V_\pi^H$ is even, then $\op{Eig}_{\Delta^\C} (\lambda^\pi)$ contains $2$-times $\pi$, so it not $G$-irreducible; 
	if $\pi\in\widehat G_H$ is of complex type, then $\pi^*\in\widehat G_H$, $\pi^*\not\simeq\pi$, and  $\lambda^\pi=\lambda^{\pi^*}$, thus $\op{Eig}_{\Delta^\C} (\lambda^\pi)$ contains $\pi\oplus\pi^*$, so it is not $G$-irreducible. 
See \cite[Rem.~2.5]{Schueth17} for a similar explanation in the Lie group case (i.e.\ $H=\{e\}$). 

However, it is not difficult to obtain an analogous result to Proposition~\ref{prop:PetreccaRoser} for the complex case, including the condition (a). 
Therefore, Theorem~\ref{thm:standard-fails} and Theorem~\ref{thm1:main} still hold for $\Delta^\C$. 
\end{remark}

\bibliographystyle{plain}

\end{document}